\documentclass[journal,twoside,web]{ieeecolor}
\usepackage{generic}
\usepackage{cite}
\usepackage{amsmath,amssymb,amsfonts}
\usepackage{algorithmic}
\usepackage{graphicx}
\usepackage{algorithm,algorithmic}
\usepackage{mathtools,mathbbol,dsfont}
\usepackage{bm}
\usepackage{textcomp}
\usepackage{subcaption}
\usepackage[all]{xy}
\usepackage{chemarrow}
\usepackage{color}
\usepackage{mathrsfs,array}
\makeatletter

\newcommand{\Rmnum}[1]{\expandafter\@slowromancap\romannumeral #1@}
\makeatother
\newtheorem{theorem}{Theorem}
\newtheorem{corollary}{Corollary}
\newtheorem{lemma}{Lemma}
\newtheorem{definition}{Definition}
\newtheorem{proposition}[theorem]{Proposition}
\newtheorem{remark}{Remark}

\usepackage{hyperref}
\hypersetup{hidelinks=true}
\usepackage{textcomp}
\def\BibTeX{{\rm B\kern-.05em{\sc i\kern-.025em b}\kern-.08em
    T\kern-.1667em\lower.7ex\hbox{E}\kern-.125emX}}
\begin{document}
\title{Automatic Implementation of Neural Networks through Reaction Networks---Part \Rmnum{2}: Error Analysis}
\author{Yuzhen Fan, Xiaoyu Zhang, Chuanhou Gao, \IEEEmembership{Senior Member, IEEE}, and Denis Dochain
\thanks{This work was funded by the National Nature Science Foundation of China under Grant No. 12320101001, 12071428, and 62303409, the China Postdoctoral Science Foundation under Grant No. 2023M733115. This work is an extension of our earlier paper in the 4th IFAC Workshop on Thermodynamics Foundations of
Mathematical Systems Theory, Jul. 25-27, 2022, Canada.} 
\thanks{Y. Fan and C. Gao are with the School of Mathematical Sciences, X. Zhang is with the College of Control Science and Engineering, Zhejiang University, Hangzhou, China (e-mail: yuzhen$\_$f@zju.edu.cn, gaochou@zju.edu.cn (correspondence), Xiaoyu\_Z@zju.edu.cn). }
\thanks{D. Dochain is with 
ICTEAM, UCLouvain, B ˆatiment Euler, avenue Georges Lemaˆıtre 4-6, 1348 Louvain-la-Neuve, Belgium (e-mail: denis.dochain@uclouvain.be).}}

\maketitle

\begin{abstract}
This paired article aims to develop an automated and programmable biochemical fully connected neural network (BFCNN) with solid theoretical support. In Part \Rmnum{1}, a concrete design for BFCNN is presented, along with the validation of the effectiveness and exponential convergence of computational modules. In this article, we establish the framework for specifying the realization errors by monitoring the errors generated from approaching equilibrium points in individual modules, as well as their vertical propagation from upstream modules and horizontal accumulation from previous iterations. Ultimately, we derive the general error upper bound formula for any iteration and illustrate its exponential convergence order with respect to the phase length of the utilized chemical oscillator. The numerical experiments, based on the classification examples, reveal the tendency of total errors related to both the phase length and iteration number.
\end{abstract}

\begin{IEEEkeywords}
biochemical fully connected network, exponential convergence, error analysis, equilibrium approximation, error propagation
\end{IEEEkeywords}

\section{Introduction}
\label{sec:introduction}

\IEEEPARstart{C}{ells} in nature exhibit remarkable adaptability to external stimuli through sophisticated biomolecular interactions, including gene regulatory networks and signal pathway networks. In synthetic biology, engineering principles are applied to design biological components and construct intricate circuits, aiming to emulate life behaviors and program cells for diverse applications such as medical care, energy, and the environment \cite{qian2018programming}. Notably, recent progress in this field emphasizes a "system level" design, facilitating the creation of sufficiently elaborate biochemical reaction circuits to perform a wider range of tasks, such as robustness enhancement \cite{qian2021robustness, xiao2018robust,aoki2019universal,frei2022genetic,filo2022hierarchy}, information processing \cite{qian2011neural,cherry2018scaling,
zhang2020cancer} and more. 

Molecular computational modules are crucial components for performing information processing objectives within engineered biological systems. To design biochemical circuits serving various computational purposes, chemical reaction networks (CRNs), a widely recognized mathematical model describing biomolecular interactions, are consistently regarded as the bridge. This is attributed to the potential DNA strand displacement reactions implementation of the abstract CRN model \cite{soloveichik2010dna}, and the capacity of CRNs equipped with mass-action kinetics (MAK), represented as polynomial ordinary differential equations (PODEs), to compile any computation, as asserted by Turing completeness \cite{fages2017strong}. Significant strides have been recently achieved toward the goal of creating programmable biological systems to execute intricate information-processing tasks \cite{anderson2021reaction,arredondo2022supervised,blount2017feedforward,lakin2023design,samaniego2021signaling,vasic2020crn++,vasic2022programming,vasic2020deep,FAN20227,fan2023automatic}. These achievements have been realized through programming artificial neural networks (ANNs) using biochemical reaction network (BCRN) systems, which aim to leverage the established computational framework and training mechanism of ANNs, enabling biological circuits to efficiently solve complicated problems like classification and decision-making, etc., and exhibit intelligent learning ability. 

Generally speaking, the BCRN realization of computational modules in FCNN relies on approximating calculation results through the limiting steady states (LSSs) of the corresponding reaction modules. Theoretically, achieving asymptotically stable equilibrium points is possible over infinite time, while obtaining arbitrary computation results is feasible within finite time. The phase length of the chemical oscillator \cite{arredondo2022supervised,lakin2023design,fan2023automatic} which signifies the operation time of each module is exactly limited. On top of that, the cascade link of multiple biochemical computational modules and multiple iterations in the training process will result in the propagation of errors within each module caused by such equilibrium approximation. Consequently, there is a common issue with such BCRN realization of neural networks, as previously discussed in \cite{FAN20227} i.e. whether the error compromises the performance of biochemical neural networks, especially the original convergence of the training algorithm used in FCNN. However, a comprehensive characterization of the error is not found in the literature.
In Part \Rmnum{1} \cite{fan2023automatic}, we devise the biochemical fully-connected neural network (BFCNN) system endowed MAK that is the first complete realization of a fully-connected neural network (FCNN) comprising the feedforward module, learning module, and all connection modules. We also demonstrated that all computational modules permit exponentially stable equilibrium points. This guarantees a rapid convergence of the reaction dynamics from the initial values to the equilibrium points. Certainly, this approach effectively minimizes errors as much as possible, while it does not eliminate them entirely. 

In this paper, we perform an in-depth error analysis on the BFCNN. Our contribution lies in first showing the upper bound of the cumulative error up to any iteration number on weights updating and then demonstrating that the convergence order of errors concerning the finite phase length is exponential. This implies that if the training algorithm for neural networks permits the weight matrices, starting from certain initial values, to approach the solution after multiple iterations progressively, the actual concentrations of weight species in BFCNN will inherently possess the convergence capability of approaching the ideal solution with any given precision within a finite timeframe. 

The organization of this paper is as follows. Section \ref{sec:pre} covers preliminaries about CRN and FCNN. Section \ref{sec:Error in 1st} considers the error accumulated during the first iteration to illustrate impacts from individual modules and vertical propagation. Section \ref{sec:error_m iteration} investigates the horizontal propagation of errors by incorporating influence from previous iterations, thereby presenting the general error upper bound formula for any iteration. The trends of the upper bound concerning the phase length and iteration number are validated in Section \ref{sec:numeri}. Finally, Section \ref{sec:conclusion} concludes this article.

\noindent \textbf{Notations:} $\mathbb{R}^n, \mathbb{R}^n_{\geq0}, \mathbb{R}^n_{>0}, \mathbb{Z}^n_{\geq0}$ denote $n$-dimensional real space, nonnegative space, positive space and nonnegative integer space, respectively. $\mathscr{C} (\cdot;*)$ is the set of continuous differentiable functions from $\cdot$ to $*$. For any real value function $f(\cdot)$ and a matrix $A = (a_{ij}) \in \mathbb{R}^{n \times n}$, $f(A) = (f(a_{ij})) \in \mathbb{R}^{n \times n} $ denotes $f(\cdot)$ acts on any element of $A$. Besides, $a_{\cdot j}$ indicates the $j$th column, and $a_{i\cdot}$ represents the $i$th row in $A$. The norm $\Vert \cdot \Vert$ appearing in this paper is regarded as $\infty$-norm in Euclidean space. $\mathds{1}$ and $\mathds{O}$ represent a dimension-suited column vector and matrix with all elements being $1$ and $0$, respectively. 


\section{Preliminaries}
\label{sec:pre}
In this section, we formally introduce the basic concepts of CRNs \cite{feinberg1972complex,feinberg1979lectures,horn1972general} and present an overview of FCNN with a specific feedforward structure, along with the mini-batch gradient descent (MBGD) training algorithm. 
\subsection{Chemical Reaction Network}
Consider a CRN with $n$ species that interact with each other through $r$ reactions
\begin{definition}[CRN]
A CRN consists of three finite sets, i.e., \textit{species set:} $\mathcal{S}=\{X_{1}, \ldots,X_{n}\}$ denotes the subjects participated in reactions; \textit{complex set:} $\mathcal{C}=\bigcup_{j=1}^{r}\{v_{.j},v'_{.j}\}$ with $v_{.j},v'_{.j}\in \mathbb{Z}_{\geq 0}^n$ where each complex is the linear combination of species $\bigcup_{i=1}^{n}X_i$ and $v_{ij}$ is the stoichiometric coefficient of $X_i$; \textit{reaction set:} $\mathcal{R}=\bigcup_{j=1}^{r}\{v_{.j}\to v'_{.j}\}$ describe interactions between species, often referred to as a triple $(\mathcal{S},\mathcal{C},\mathcal{R})$.
\end{definition}

Based on the above definition, the $j$th reaction is written as
\begin{equation}
	\sum_{i=1}^{n} v_{ij} X_{i} \longrightarrow \sum_{i=1}^{n} v'_{ij} X_{i},
	\label{eq:1}
\end{equation} 
where $v_{.j}=(v_{1j},...,v_{nj})^\top, v'_{.j}=(v'_{1j},...,v'_{nj})^\top$ are called \textit{reactant complex} and \textit{product complex}, respectively. The dynamics of $(\mathcal{S},\mathcal{C},\mathcal{R})$ that captures the change of the concentration of each species, labeled by $x\in\mathbb{R}^n_{\geq 0}$, may be reached if a $r$-dimensional vector-valued function $\mathscr{K}\in \mathscr{C}(\mathbb{R}^n,\mathbb{R}^r)$ is defined to evaluate the reaction rates and the balance
law is further utilized, written as
\begin{equation}
    \frac{dx(t)}{dt} = \Gamma \mathscr{K}(x),~~x \in \mathbb{R}^n_{\geq0}.
    \label{eq:2}
\end{equation}
where $\Gamma\in\mathbb{Z}^{n\times r}$ is the \textit{stoichiometric matrix} with the $j$th column $\Gamma_{\cdot j}=v'_{.j}-v_{.j}$. Usually, \textit{mass action kinetics} is used to characterize the reaction rate, which induces the rate of the $j$th reaction as $\mathscr{K}_j(x)=  k_j x^{v.j} \triangleq k_j\prod_{i=1}^{n} x_{i}^{v_{ij}}$,
where $k_j>0$ represents the rate constant. The CRN system endowed with MAK offers PODEs, often termed as \textit{mass action system} (MAS) and denoted by the quad $(\mathcal{S},\mathcal{C},\mathcal{R},k)$. 

\subsection{Fully Connected Neural Network}
Assume a FCNN with an input layer, a hidden layer, and an output layer, and let it perform a task for the data set $\mathbb{D}=\{(x^i,d^i)\}_{i=1}^p$ with $p$ to represent the size of $\mathbb{D}$. The training process utilizes the common gradient descent-based error backpropagation algorithm (GDBP) with the mini-batch strategy (i.e., at each iteration only part of samples are fed to update weights), and the activation function takes the \textit{Sigmoid function}. 
For simplicity but without the loss of generality, we fix the structure of FCNN to be \textit{two-two-one} nodes in every related layer to exhibit this process.
\subsubsection{Feedforward Propagation}
Denote the sample matrix by $\chi\in\mathbb{R}^{3\times p}$ with the $l$th column to be $\chi_{.i}=(x^i_1,x^i_2,d^i)^\top$, the \textit{input matrix} by $\Xi\in\mathbb{R}^{3 \times \tilde{p}}_{\geq 0}$ with $\Xi_{3.}=\mathds{1}^\top$ to handle dumb nodes corresponding to bias in the input layer and $\tilde{p}$ ($\tilde{p}\leq p$) to represent the mini-batch size, and the \textit{weight matrices} connecting input-hidden layers and hidden-output layers by $\mathcal{W}_1\in\mathbb{R}^{2\times 3}$ and $\mathcal{W}_2\in\mathbb{R}^{1\times 3}$ with their last columns to represent the corresponding bias. Then the feedforward computation process at the $m$th iteration follows
\begin{gather}\label{eq:weightsum}
\begin{split}
\mathcal{N}^m&=\mathcal{W}_1^m\cdot \Xi^m,\Upsilon^m(\mathcal{W}^m_1) =f(\mathcal{N}^m),\tilde{\Upsilon}^m=
    \begin{pmatrix}
        \Upsilon^m\\
        \mathds{1}^\top
\end{pmatrix},\\
\tilde{\mathcal{N}}^m&=\mathcal{W}_2^m\cdot \tilde{\Upsilon}^m, y^m(\mathcal{W}_2^m) =f(\tilde{\mathcal{N}}^m),
 \end{split}
\end{gather}
where $\mathcal{N}\in \mathbb{R}^{2\times \tilde{p}}$, $\Upsilon\in\mathbb{R}^{2\times \tilde{p}}$ is the \textit{hidden matrix}, $\tilde{\Upsilon}\in\mathbb{R}^{3\times \tilde{p}}$ is constructed to add a row of $1$ to $\Upsilon$ to handle dumb nodes in the hidden layer, $\tilde{\mathcal{N}}\in \mathbb{R}^{1\times \tilde{p}}$, and $y\in\mathbb{R}^{1\times \tilde{p}}$ is the \textit{output matrix}. We further can compute the training error $\mathcal{E}^m(\mathcal{W}_1^m, \mathcal{W}_2^m)=\frac{1}{2}(\delta^\top-y^m)(\delta^\top-y^m)^\top,$
where $\delta=(d^1,...,d^{\tilde{p}})^\top$. This finishes the feedforward computation of the $m$th iteration with $\tilde{p}$ samples.
\subsubsection{Backward Propagation}
This process serves to update weights according to GDBP, which takes
\begin{gather}
\begin{split}
\mathcal{W}^{m+1}&=\mathcal{W}^{m} + \Delta\mathcal{W}^m
       \label{eq:GD}
\end{split}
\end{gather}
with $\mathcal{W}=
    \begin{pmatrix}
        \mathcal{W}_1\\
        \mathcal{W}_2
\end{pmatrix}$, $\eta \in \left(0,1\right]$ to be the learning rate and 
\begin{gather}\label{eq:weight_update}
\begin{split}
\Delta \mathcal{W}_{1_{ij}}^m &= -\eta \frac{\partial \mathcal{E}^m}{\partial  \mathcal{W}^m_{1_{ij}}} = -\eta \sum^{\tilde{p}}_{l=1} e^m_l \frac{\partial f}{\partial n^m_{3l}} \mathcal{W}^m_{2_{1i}} \frac{\partial f}{\partial n^m_{il}} \Xi_{jl}^m, \\
\Delta \mathcal{W}_{2_{1j}}^m &=-\eta \frac{\partial \mathcal{E}^m}{\partial  \mathcal{W}^m_{2_{1j}}} = -\eta \sum^{\tilde{p}}_{l=1} e^m_l \frac{\partial f}{\partial n^m_{3l}} \tilde{\Upsilon}^m_{jl},\\
&~~~~i=1,2;~j=1,2,3.
\end{split}
\end{gather}
Here, $e_l^m=(\delta^{\top} - y^m)_{1l},~(n^m_{1l}, 
n^m_{2l})^{\top} = \mathcal{N}^m_{\cdot l},~n^m_{3l} = \tilde{\mathcal{N}}^m_{1l}$. 
Then, the training is terminated if $\vert e^m_l \vert$ is less than the preset threshold for all $l$, otherwise, it continues with another group of $\tilde{p}$ samples from $\mathbb{D}$. 


\section{Composition of Errors}
\label{sec:com_of_error}
This section provides a comprehensive overview of the error composition throughout the entire process. The BFCNN encompasses five primary modules, i.e., the assignment module, feedforward module, judgment module, learning module, and clear-out module where some modules incorporate several sub-computational reaction modules. The chemical oscillator, consisting of oscillatory reactions with rate constant $k_o$, governs their sequential execution by assigning distinct oscillatory species as catalysts for different reaction modules, thereby reactions within each module occur exclusively during the non-zero phase of their respective catalysts. Suppose that every non-zero phase in the oscillator will last for an adjustable time $T$ (called phase length) in practice.
If $s \in \mathbb{Z}$ denotes the number of reaction modules in which any species $X$ participates during an individual iteration, its concentration $x(sT)$ refers to the corresponding endpoint of its dynamical concentration flow. Notably, we omit the time interval during which species $X$ is not involved in reactions. Hereafter, the definition of all concentration variables here follows the conventions established in \cite{fan2023automatic}. For clarity, we regard $x(sT)$ in BFCNN (real system) and $x$ in FCNN (reference model) as the real computation result and the corresponding standard computation result, respectively. Then, we illustrate the following two types of error composition for the total error.

\subsubsection{Current error $\bar{\epsilon}_i^m$}
Considering each sub-computational reaction module $\mathcal{M}_i$ occurs in certain phase with the output species $X_{o_i}$, its standard computation result $x_{o_i}$ is exactly the theoretical LSS $\bar{x}_{o_i} = \lim_{t \to \infty} x_{o_i}(t)$ if there are no other errors. Caused by the finite phase length $T$ for $\mathcal{M}_i$, we describe the error $\vert x_{o_i}(T) -\bar{x}_{o_i} \vert <\bar{\epsilon}_i^m(T)$ as \textit{current error} with $\bar{\epsilon}_i^m(T)$ being \textit{current error bound}.

\subsubsection{Accumulative error $\tilde{\epsilon}^m_i$} The equilibrium $\bar{x}_{o_i}$ and its convergence characterization depend on the initial values and parameters of the reaction system in $\mathcal{M}_i$, which could be the LSSs of upstream modules and previous iterations. Thus, this component arises from cumulative current errors propagated vertically from upstream modules to $\mathcal{M}_i$ and horizontally from previous iterations to the $m$th iteration, as shown in Fig.\ref{fig:4}. Both contribute to \textit{accumulative error} which is described as $\vert \bar{x}_{o_i} - x_{o_i} \vert <\tilde{\epsilon}^m_i$ with $\tilde{\epsilon}^m_i$ being \textit{accumulative error upper bound} or solely altering the current error. 

\begin{figure}[!t]
\centerline{\includegraphics[width=\columnwidth]{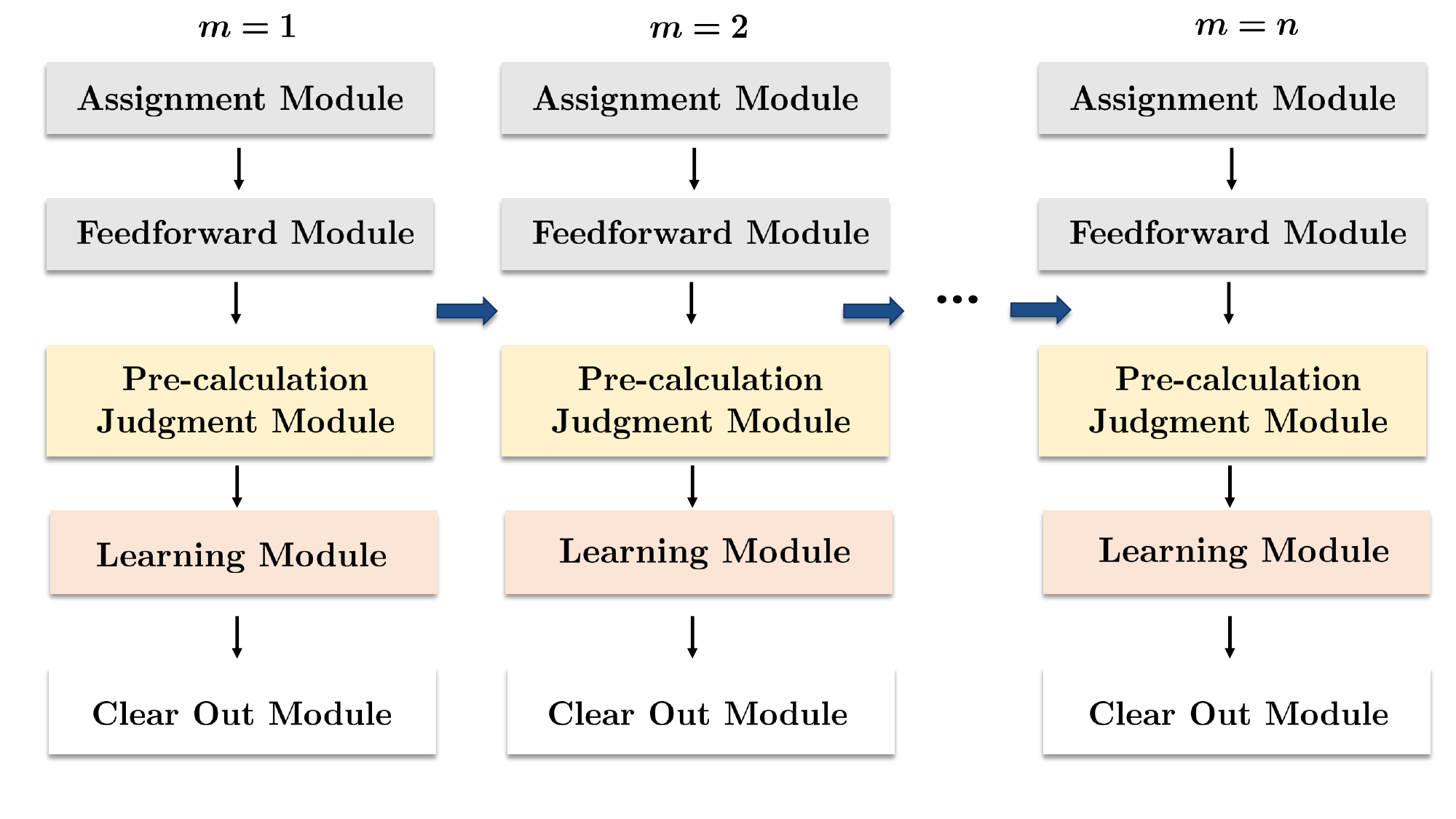}}
\caption{The propagation of the current error. The vertical thin black arrow denotes the vertical propagation, and the horizontal thick arrow describes the horizontal propagation.}
\label{fig:4}
\end{figure}



Up to the $m$th iteration, the upper bound for the total deviation of the output in $\mathcal{M}_i$ becomes $ \vert x_{o_i}(T) - x_{o_i} \vert < \epsilon_i^m = \bar{\epsilon}_i^m+\tilde{\epsilon}_i^m$. Consequently, errors accumulate in the weight matrices of each iteration, causing the actual iterative relationship of weights to differ from that in \eqref{eq:GD}. We note that, in BFCNN, the difference of equilibrium concentration matrices $\bar{w}^m_+ - \bar{w}^m_-$ of positive and negative weight/bias species computes the weight matrix $\mathcal{W}^m = \mathcal{W}_+^m-\mathcal{W}_-^m$ based on the dual-rail encoding. Considering the positive part, one can have $\bar{w}^{m+1}_{+} = w^m_+(T)-\eta \nabla \mathcal{E}_+^m(w^m_+(T), w^m_-(T))$ with $\nabla \mathcal{E}^m = \nabla \mathcal{E}^m_{+}-\mathcal{E}^m_{-}$ denoting the gradient function of $\mathcal{E}$.
Then, 
\begin{equation}
\begin{split}\label{eq:gen_w_error}
\Vert & w^{m+1}_+(T) - \mathcal{W}^{m+1}_+ \Vert  < \bar{\epsilon}^{m+1}_{w_+} + \Vert w^{m}_+(T) - \mathcal{W}^m_+  \Vert \\
& + \eta \Vert  \nabla \mathcal{E}^m_+(w^m_+(T), w^m_-(T))-  \nabla \mathcal{E}^m_+(\mathcal{W}^m_+, \mathcal{W}^m_-) \Vert \\
& < \bar{\epsilon}^{m+1}_{w_+} + \epsilon^m_{w_+} + \eta \delta^{m+1}_+,
\end{split}
\end{equation}
where $\mathcal{W}^m_{\pm} \in \mathbb{R}^{3 \times 3}_{\geq 0}$ indicate the standard positive and negative part of $\mathcal{W}^m$, and $\delta^{m+1}_+$ originates from computing the gradient in $m$th iteration ($\epsilon^m_{w_+}$ may contain errors of assigning $w^m_+(T)$ relying on design). The iteration formula $\tilde{\epsilon}^{m+1}_{w_+}= \tilde{\epsilon}^m_{w_+} +\bar{\epsilon}^m_{w_+} + \eta \delta^{m+1}_+ $ is derived from \eqref{eq:gen_w_error} and the rough expression can be written as 
\begin{equation*}
\tilde{\epsilon}^{m+1}_{w_+} = \sum_{i=0}^{m} (\Bar{\epsilon}^j_{w_+} + \eta \delta^{j+1}_+),
\end{equation*}
where $\bar{\epsilon}^0_{w_+}=0$ since initial weight concentrations are given in advance and not subject to calculation. In particular, $\bar{\epsilon}^j_{w_+} + \eta \delta^{j+1}_+$ can be interpreted as the error accumulated vertically from all upstream modules up to the uBCRN system at $j$th iteration. The summation from zero to $m$ shows the horizontal propagation. A similar argument can be applied in deducing $\bar{\epsilon}^{m+1}_{w_-}, \tilde{\epsilon}^{m+1}_{w_-}$ and then obtaining $\epsilon^{m+1}_{w}$ by $\epsilon^{m+1}_{w_+}+\epsilon^{m+1}_{w_-}$.


Roughly speaking, if the training process is confirmed convergent, namely for $\forall \epsilon$, $\exists K \in \mathbb{Z}_{>0}$ s.t. $\Vert \mathcal{W}^m-\mathcal{W}^* \Vert$ \textless $\epsilon$ for all $m \geq K$, and the total error $\epsilon^m_{w}$ caused by BCRN dynamics is also small enough and vanishes with increasing finite time $T$, the actual weight concentration matrix sequence can approach the accurate solution in MBGD with any given precision. In the following sections, we will proceed to estimate errors in all phases specific to the content in $\bar{\epsilon}^m$ and $\delta^m$ and show the error upper bound and the error convergence order. 



\section{Error Estimation in the First Iteration}
\label{sec:Error in 1st}
In this section, our attention is directed towards the error vertically propagated from upstream modules to the learning module, and we present the total deviation upper bound in weight updating for $m=1$. Suppose that no errors arise during the injection of any species $X$ with the corresponding concentrations into reactors, so $x^1(0)$ is accurate. Hereafter, $\bar{x}$ and $x$ appearing after the specific module indicate the variables of the current module unless otherwise specified as $\{x\}_{\mathcal{O}_i}$, and parameters in scalar estimation vary based on $l$ ($l$ denotes input size throughout this paper) but are denoted without explicitly specifying $l$ for convenience.


\subsection{Assignment Module Error}
The assignment module (aBCRN) consists of three sub-computational reaction modules $\mathcal{M}^a_1, \mathcal{M}^a_2, \mathcal{M}^a_3$ occupying three phases $\mathcal{O}_1, \mathcal{O}_3, \mathcal{O}_5$ of one period, and each sub-computational module is the MAS demonstrated to permit globally exponentially stable (GES) equilibrium points relative to its stoichiometry compatibility class in Part \Rmnum{1}. For $\forall i \in \{1,...,p\},~l\in \{1,...,\tilde{p}\}$, species of interest are, the sample species set $\{X^i_1, X^i_2, D^i\}_{i=1}^p$, input species set $\{S^l_1, S^l_2, S^l_3\}$, order species set $\{\tilde{C}^l_i\}$, and auxiliary order species set $\{\tilde{C}^l_i\}$ with their concentration vectors $\mathcal{X}_{.i}(t)=(x^i_1(t),x^i_2(t),d^i(t))^\top$, $s^l(t)=(s^l_{1}(t),s^l_{2}(t),s^l_3(t))^{\top}$, $c^l(t) = (c^l_1(t),\cdots,c^l_p(t))^{\top}$, $\tilde{c}^l(t) = (\tilde{c}^l_1(t),\cdots,\tilde{c}^l_p(t))^{\top}$, respectively. Fig.\ref{fig:5} shows that, when $m=1$, only current errors appear in $\mathcal{M}^a_1, \mathcal{M}^a_2$, and the accumulative error in $\mathcal{M}^a_3$ caused by the vertical propagation from $\mathcal{M}^a_2$ is inevitable since $c^l(T), \tilde{c}^l(T)$ become its initial condition. 
\begin{figure}[!t]
\centerline{\includegraphics[width=\columnwidth]{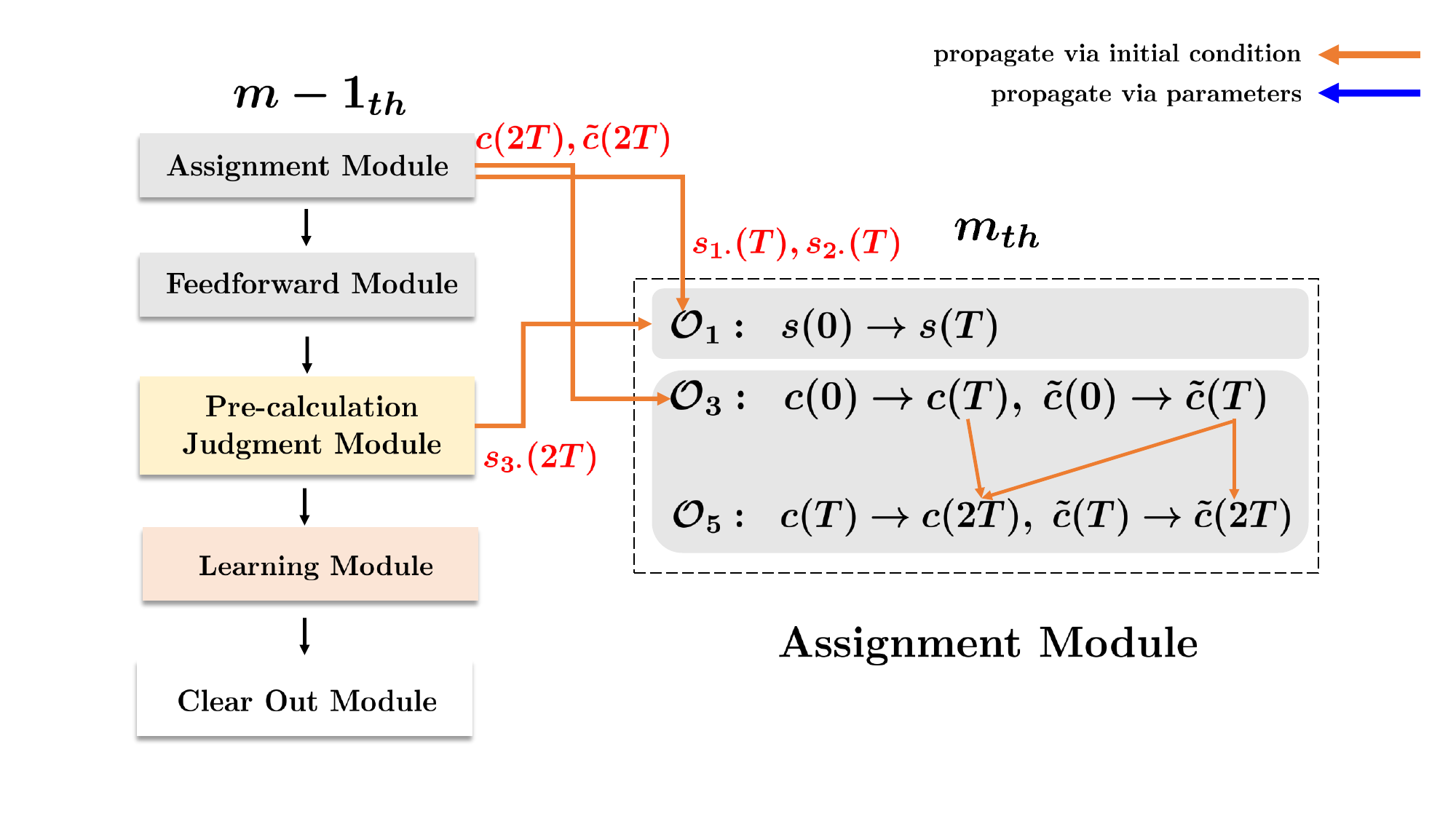}}
\caption{Schema of the propagation of error through initial concentrations from the $m-1$th to the $m$th iteration in the Assignment module.}
\label{fig:5}
\end{figure}
Thus, given $s(t) \in \mathbb{R}^{3 \times \tilde{p}}_{\geq 0}$, $c(t), \tilde{c}(t) \in \mathbb{R}^{p \times \tilde{p}}_{\geq 0}$ with $s^l(t), c^l(t), \tilde{c}^l(t)$ being their columns, respectively, and combining $\mathcal{X}(t) \in \mathbb{R}^{3 \times p}$, we have the following results for the case of $m=1$.
\begin{proposition}\label{prop:assign_1}
Let $s=(\chi_{.{1}},...,\chi_{.\tilde p})^{\top}$,  $\tilde{c}=\mathds{O}_{p \times \tilde{p}}$, $c = (e_{1+\tilde{p}},\cdots, e_{2\tilde{p}})^\top$ where $e_l \in \mathbb{R}^{1 \times p}$ is a unit row vector with the $l$th element to be 1 and $I_l\triangleq\{l,\tilde{p}+l,...,p-2\tilde{p}+l\}$, then there are positive coefficients $m_1, m_3, \tilde{m}_3$ such that the total error upper bound after the assignment module become
\begin{equation}
    \begin{split}
\Vert s(T) - s &\Vert \leq  m_1 e^{-T}\\
\Vert c(2T) -c\Vert \leq m_3 e^{-kT},&~ \Vert \tilde{c}(2T) -\tilde{c} \Vert  \leq \tilde{m}_3 e^{-kT},
    \end{split}
\label{eq:prop_assi}
\end{equation}
where $k>0$ is the rate constant in $\mathcal{M}^a_2, \mathcal{M}^a_3$.
\end{proposition}
\begin{proof}
Recall that for $o \in I_l,~ j \notin I_l \bigcup \{p-\tilde{p}+l\}$, the dynamical equations of $\mathcal{M}^a_1, \mathcal{M}^a_2, \mathcal{M}^a_3$ takes
\begin{gather*}
\begin{split}
\mathcal{M}^a_1:~~ \dot{s}^l(t) & =\mathcal{X}(t) ({c^l} (t))^\top - s^l(t),\\
\mathcal{M}^a_2:~~~ \dot{c}(t) & = -kc(t),~ \dot{\tilde{c}}(t) = kc(t),\\
\mathcal{M}^a_3:~~ \dot{c}^l_{l}(t) & = k \tilde{c}^l_{p-\tilde{p}+l}(t), ~\dot{c}^l_{o+\tilde{p}}(t) = k \tilde{c}^l_{o}(t), \\
\dot{c}^l_j(t) & = k \tilde{c}^l_j(t),~~
    \dot{\tilde{c}}^l_i(t)= -k \tilde{c}^l_i(t).
\end{split}
\end{gather*}
We obtain $m_1>0$ for $\mathcal{M}^a_1$ through
\begin{align*}
\Vert s(t)-s \Vert & \leq \Vert s(t)-\bar{s} \Vert + \Vert \bar{s}- s \Vert = \Vert s(0)-\mathcal{X}(0)c(0) \Vert e^{-t}.
\end{align*}
Here, $\bar{s} = \mathcal{X}(0)c(0) = s $. Thus, let $t=T$, then $m_1= \Vert s(0)- \mathcal{X}(0)c(0) \Vert$. A similar argument applies to the order species and auxiliary order species in $\mathcal{M}^a_2$ to get the total error
\begin{align*}
\Vert c(T)- \mathds{O}_{p \times \tilde{p}} \Vert  = m_2 e^{-kT},~ \Vert \tilde{c}(T)- c(0) \Vert = m_2 e^{-kT},
\end{align*}
where $m_2 = \Vert c(0) \Vert$. The dynamical flow of species $\{C^l_i\}$, $\{\tilde{C}^l_i\}$ in phase $\mathcal{O}_5$ is actually govern by $\mathcal{M}^a_3$ and runs from $T$ to $2T$. Thus, we obtain the current error estimation for $\forall ~l$ 
\begin{align*}
    \Vert \tilde{c} (2T) - \bar{\tilde{c}} \Vert  = & \Vert \tilde{c} (T)  e^{-kT} - \mathds{O}_{p \times \tilde{p}}  \Vert = \Vert \tilde{c} (T)  \Vert e^{-kT}, \\
    \Vert  c^l(2T) - \bar{c}^l  \Vert  = & \Vert  \mathcal{K}^l \tilde{c}^l(T)(1-e^{-kT}) + c^l (T)  \\
    & - ( \mathcal{K}^l \tilde{c}^l(T)  + c^l (T))\Vert  = \Vert \mathcal{K}^l \tilde{c}^l(T) \Vert e^{-kT},
\end{align*}
where the coefficient matrix $\mathcal{K}^l \in \mathbb{Z}^{p \times p}_{\geq 0}$ is
\begin{equation*}
\mathcal{K}^l=
{\scriptsize
\begin{pmatrix}
E_{l-1}  & \mathds{O}_{(l-1) \times (p-l+1)} \\
\mathds{O}_{\tilde{p} \times (l-1)} & \mathcal{K}_1 \\
\mathds{O}_{\tilde{p} \times (l-1)} & \mathcal{K}_2 \\
 \vdots & \vdots \\
\mathds{O}_{(\tilde{p}-l) \times (l-1)} & \mathcal{K}_{p/\tilde{p}} \\
\end{pmatrix}},
\end{equation*}
where $E_{l-1}$ is a $(l-1)$-dimensional unit matrix and $\mathcal{K}_1$ is
\begin{equation*}
{\scriptsize \bordermatrix{%
 & C_l &  & C_{l+\tilde{p}} & & C_{l+p-\tilde{p}} & \cdots  \cr
~~R_{l} & 0   &  \cdot \cdot & 0 & \cdots & 1 & \mathds{O}_{1 \times (\tilde{p}-l)} \cr
  & \mathds{O}_{(\tilde{p}-1) \times 1}   &  E_{\tilde{p}-1} & \mathds{O}_{(\tilde{p}-1) \times 1} & \cdots  &\mathds{O}_{(\tilde{p}-1) \times 1} & \mathds{O}_{(\tilde{p}-1) \times (\tilde{p}-l)} \cr}},
\end{equation*}
for $\forall b \in \{1,\cdots, \frac{p}{\tilde{p}}-2\}$,  $\mathcal{K}_{b+1}$ is written as 
\begin{equation*}
{\scriptsize \bordermatrix{%
& \cdots & C_{l+(b-1)\tilde{p}} &  &  C_{l+b\tilde{p}} & & C_{l+(b+1)\tilde{p}} & \cdots  \cr
  ~~ R_{l+b\tilde{p}} & \mathds{O}_b^1 & 1 & \cdot \cdot & 0 & \cdots & 0 & \cdots \cr
  & \mathds{O}_b^2 & \mathds{O}_{(\tilde{p}-1) \times 1}  & \cdot \cdot & \mathds{O}_{(\tilde{p}-1) \times 1}  &E_{\tilde{p}-1} & \mathds{O}_{(\tilde{p}-1) \times 1}    & \cdots \cr}},
\end{equation*}
and when $b=\frac{p}{\tilde{p}}-1$, $\mathcal{K}_{b+1}$ becomes
\begin{equation*}
{\scriptsize \bordermatrix{%
& \cdots & C_{l+(b-1)\tilde{p}} &  &  C_{l+b\tilde{p}} &   \cr
  ~~ R_{l+b\tilde{p}} & \mathds{O}_b^1   & 1 & \cdot \cdot & 0 & \cdots  \cr
  & \mathds{O}_b^2   &  \mathds{O}_{(\tilde{p}-1) \times 1}  & \cdot \cdot &  \mathds{O}_{(\tilde{p}-1) \times 1} & E_{\tilde{p}-l} \cr}}
\end{equation*}
with $\mathds{O}_b^1=\mathds{O}_{1 \times (1+(b-1)\tilde{p})}$ and $\mathds{O}_b^2=\mathds{O}_{(\tilde{p}-1) \times (1+(b-1)\tilde{p})}$.
Combining propagation error in the initial condition $\tilde{c}(T), c(T)$ for $\mathcal{M}^a_3$, we have the accumulative error induced in order species for $\forall l$
\begin{align*}
    \Vert \bar{c}^l - c_{\cdot l} \Vert & = \Vert ( \mathcal{K}^l \tilde{c}^l(T)  + c^l (T)) - \mathcal{K}^l c^l(0) \Vert \\
    & \leq \Vert  \mathcal{K}^l \Vert \Vert  c(0) \Vert e^{-kT} + \Vert  c(0) \Vert e^{-kT},
\end{align*}
and the total error upper bound
\begin{equation*}
\begin{split}
\Vert  c^l(2T) - c_{\cdot l}  \Vert & \leq  \Vert  c^l(2T) - \bar{c}^l  \Vert +  \Vert \bar{c}^l - c_{\cdot l} \Vert \\
& \leq \Vert \mathcal{K}^l \tilde{c}^l(T) \Vert e^{-kT} + (1+\Vert  \mathcal{K}^l \Vert) \Vert  c(0) \Vert e^{-kT} \\
& \leq (1+2\Vert  \mathcal{K}^l \Vert)\Vert  c(0) \Vert e^{-kT} + \Vert \mathcal{K}^l \Vert m_2 e^{-2kT}.
\end{split}
\label{eq:O_5_err1}
\end{equation*}
For $\{\tilde{C}^l_i\}$, $\mathcal{M}^a_3$ ensures that $\Vert \tilde{c} (2T) - \bar{\tilde{c}} \Vert \leq \Vert \tilde{c} (T) \Vert e^{-kT}$, and incorporating $\Vert \tilde{c}(T) \Vert \leq \Vert c(0) \Vert + m_2 e^{-kT}$ derived previously, the total error upper bound becomes $\Vert \tilde{c} (2T) - \tilde{c} \Vert  \leq \Vert c(0) \Vert e^{-kT} + m_2 e^{-2kT}$ where the additional item $m_2 e^{-2kT}$ appropriately represents the accumulative error. To facilitate subsequent expressions, according to the leading term principle, one can find that $\exists~ m_3, \tilde{m}_3 \in \mathbb{R}_{>0}$, s.t., $\Vert c(2T) - c \Vert <  m_3 e^{-kT}$ and $\Vert \tilde{c}(2T) - \tilde{c} \Vert < \tilde{m}_3 e^{-kT}$, where $\tilde{m}_3= \Vert c(0) \Vert $, $m_3=\mathop{\max}\limits_{l\in\{1,...,\tilde{p}\}}\{(1+2\Vert  \mathcal{K}^l \Vert)\Vert c(0) \Vert \}$ since $\lim_{t \to \infty} \frac{e^{-2kt}}{e^{-kt}} = 0$.
\end{proof}

Notably, in the module $\mathcal{M}^a_3$, $\bar{\tilde{c}} = \tilde{c}$, so the propagation error in initial conditions alters the original constant coefficient $\Vert \tilde{c}(T) \Vert$, rather than the equilibrium value. 

\subsection{Feedforward Module Error}
\begin{figure}[!t]
\centerline{\includegraphics[width=\columnwidth]{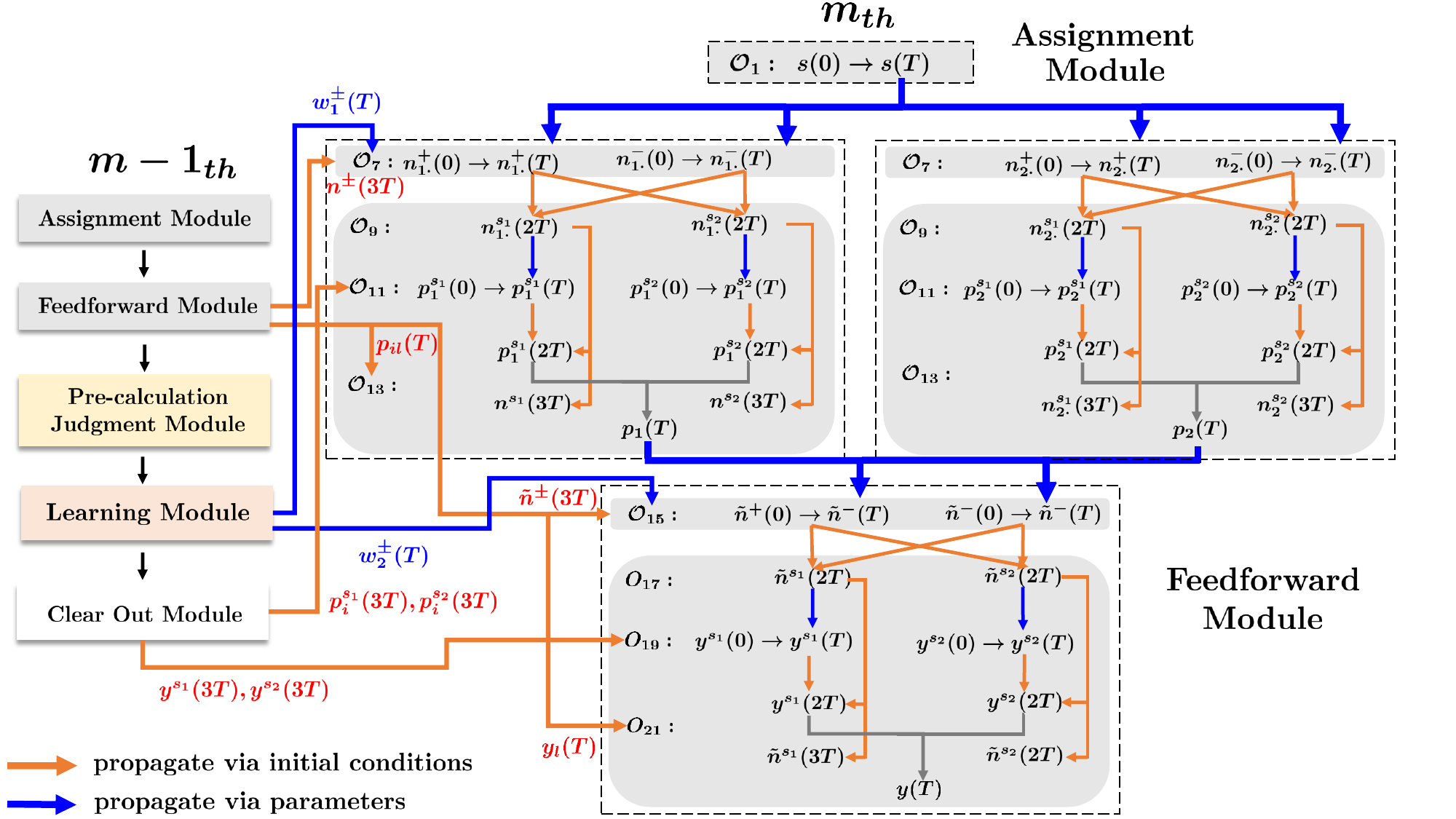}}
\caption{Schema of the propagation of error via initial concentrations and parameters from 
the $m-1$th to the $m$th iteration in the Feedforward module. The top two blocks represent the hidden layer neurons, and the bottom one denotes the output neuron.}
\label{fig:6}
\end{figure}
The feedforward module of the two-two-one FCNN comprises eight sub-computational modules occupying eight odd phases from $\mathcal{O}_7$ to $\mathcal{O}_{21}$, where linear weighted sum BCRN (lwsBCRN) designed for modules $\mathcal{M}^f_{l_1}, \mathcal{M}^f_{l_2}$ aiming to obtain $\mathcal{N}^m, \tilde{\mathcal{N}}^m$ in (\ref{eq:weightsum}) and nonlinear sigmoid activation BCRN (sigBCRN) devised for $ \mathcal{M}^f_{n^j_1}, \mathcal{M}^f_{n^j_2}$ with $j=\{1,2,3\}$ serving to compute the sigmoid activation function. Here, considering the similarity, we focus on the derivation of the first feedforward layer, i.e., $\mathcal{M}^f_{l_1}$ and $\mathcal{M}^f_{n^j_1}$. Note that $n^{\pm}(t) \in \mathbb{R}^{2 \times \tilde{p}}_{\geq 0}$ serves to denote concentrations of the net input species pair $\{N^+_{il},N^-_{il}\}^2_{i=1}$, $w^{\pm}_1(t) \in \mathbb{R}^{2 \times 3}_{\geq 0}$, $w^{\pm}_2(t) \in \mathbb{R}^{1 \times 3}_{\geq 0}$ is defined for concentrations of the weight species pair $\{W^+_{q}, W^-_{q}\}^6_{q=1}$ and the bias species pair $\{B^+_{i}, B^-_{i}\}^3_{i=1}$, and $p^{\pm}(t), p_i(t) \in \mathbb{R}^{2 \times \tilde{p}}_{\geq 0}$ indicates concentrations of hidden-layer output species pairs $\{P^{+}_{il}, P^-_{il}\}^2_{i=2}$, and 
hidden-layer output species $\{P^l_i\}^2_{i=1}$, respectively \cite{fan2023automatic}.
\subsubsection{lwsBCRN}
Both types of errors are present in $\mathcal{M}^f_{l_1}$. The current error should be obtained from the proven exponential stability of the lwsBCRN, and as displayed in Fig.\ref{fig:6}, the accumulative errors stem from the inaccuracies in parameter $s(T)$, $w^{\pm}_{1}(0)$ and initial condition $n^{\pm}(0)$. Here, only the vertical propagation from the input species needs consideration, and the estimation is presented below.


\begin{lemma}
Let $\mathcal{N}^1_+, \mathcal{N}^1_-$ denote the standard positive and negative part of $\mathcal{N}^1$, then there exist $m^+_{4}, m^-_{4} > 0$ that only rely on global initial concentrations such that the total errors for computing $\mathcal{N}^1_{\pm}$ are controlled by $m^{\pm}_{4} e^{-T}$, respectively.
\end{lemma}
\begin{proof}
Consider the dynamical equations of $\mathcal{M}^f_{l_1}$,
\begin{align} \label{eq:21}
\dot {n}^{\pm} (t) & =  w^{\pm}_{1} (t) s_a (t) - n^{\pm}(t),
\end{align}
where $s_a(t) = (s_{1\cdot}(t), s_{2\cdot}(t), \mathds{1}^{\top})^{\top}$. The current error is
\begin{gather*}
\begin{split}
\Vert n^{\pm}(T) - \bar{n}^{\pm} \Vert & =  \Vert n^{\pm}(0)-w^{\pm}_1(0)s_a(T) \Vert e^{-T}. 
\end{split}
\end{gather*}
The propagation error forces $\Vert n^{\pm}(0)-w^{\pm}_{1}(0)s_a(T) \Vert$ to be further estimated, i.e., 
\begin{equation*}
 \begin{split}
\Vert n^{\pm}(0) & - w^{\pm}_{1}(0)s_a(T) \Vert \leq \Vert n^{\pm}(0) - w^{\pm}_{1}(0) \{s_a\}_{\mathcal{O}_1} \Vert  \\
& + \Vert w^{\pm}_{1}(0) \Vert \Vert s_a(T) - \{s_a\}_{\mathcal{O}_1} \Vert, 
 \end{split}   
\end{equation*}
where $\{s_a\}_{\mathcal{O}_1} = \{((s_{1\cdot})^{\top}, (s_{2\cdot})^{\top}, \mathds{1})^{\top}\}_{\mathcal{O}_1}$. Since $\Vert A_1 \Vert \leq \Vert \left(A_1 ~ A_2 \right) \Vert$, the inequality above becomes
\begin{equation*}
 \begin{split}
\Vert n^{\pm}(0) & -w^{\pm}_{1}(0)s_a(T) \Vert \leq \Vert n^{\pm}(0) - w^{\pm}_{1}(0) \{s_a\}_{\mathcal{O}_1} \Vert  \\
& + \Vert w^{\pm}_{1}(0) \Vert \Vert s(0) - \mathcal{X}(0)c(0) \Vert e^{-T}.
 \end{split}   
\end{equation*} 
Incorporating the accumulative error on LSSs deduced as 
\begin{gather*}
    \Vert \bar{n}^{\pm} - \mathcal{N}^1_{\pm} \Vert \leq \Vert w^{\pm}_{1}(0) \Vert \Vert s(0)- \mathcal{X}(0)c(0) \Vert e^{-T},
\end{gather*}
the total error upper bound is $\Vert n^{\pm}(T) - \mathcal{N}^{1}_{\pm}\Vert \leq m^{\pm}_{4} e^{-T}$, and $m^{\pm}_{4} = \Vert n^{\pm}(0)-w^{\pm}_{1} s_a(0) \Vert +2 \Vert w^{\pm}_{1}(0) \Vert m_1$.
\end{proof}

\subsubsection{sigBCRN} 
As depicted in Fig \ref{fig:6}, the errors in calculating $\Upsilon^1$ result from the propagation through the sub-computational modules $\mathcal{M}^f_{n^j_1}$ that occur in phases $\mathcal{O}_9$, $\mathcal{O}_{11}$, $\mathcal{O}_{13}$, separately, and contain two sigBCRNs corresponding to the two nodes in the hidden layer. The module $\mathcal{M}^f_{n^1_1}$ in phase $\mathcal{O}_9$ contributes to annihilate $\{N^+_{il}\}^2_{i=1}$ and $\{N^-_{il}\}^2_{i=1}$ to obtain the difference. According to Lemma 1 in \cite{fan2023automatic}, initial concentrations of $\mathcal{M}^f_{n^1_1}$ determine the compatibility class, leading to two types of LSSs. Thus, the accumulative error here is entirely attributed to initial conditions arising from the vertical propagation from $\mathcal{M}^f_{l_1}$. We give a consolidated result to illustrate the error generated in phase $\mathcal{O}_9$.


\begin{lemma}\label{prop_O9}
Let $\vert \mathcal{N}^1_{il} \vert = \vert \{\mathcal{N}^1_{+,il}\}_{\mathcal{O}_7} - \{\mathcal{N}^1_{-,il} \}_{\mathcal{O}_7}\vert$ denotes the absolute value of standard net input. We can find $m_4 > 0$, $\{m^{'i}_j, \alpha^{'i}_j\} > 0$ with $i \in \{1,2\}, j \in \{1,2\}$ such that 
\begin{gather*}
\begin{split}
\vert n^{s_1}_{il}(2T) -\vert \mathcal{N}^1_{il} \vert \vert & \leq m_4 e^{-T} + m^{'i}_j e^{-\alpha^{'i}_j}, \\
\vert n^{s_2}_{il}(2T) - 0 \vert & \leq m^{'i}_j e^{-\alpha^{'i}_j},
\end{split}
\end{gather*}
where $i=1,2$, and if $\mathcal{N}^1_{il} > 0$ $j=1, s_1 = +, s_2=-$, otherwise $j=2, s_1 = -, s_2=+$.
\end{lemma}
\begin{proof}
The dynamical equations of $\mathcal{M}^f_{n^1_1}$ takes $\dot{n}^+_{il} (t) = \dot{n}^-_{il}(t) = -n^+_{il}(t) n^-_{il}(t)$. In case of $\mathcal{N}^1_{il} > 0$, $\{\bar{n}^+_{il}-\bar{n}^-_{il}\}_{\mathcal{O}_7} = \mathcal{N}^1_{il}$, and $\bar{n}^+_{il}$ in $\mathcal{M}^f_{l_1}$ is designed to maintain positive. Thus, the accumulative error is described as 
\begin{gather*}
\begin{split}
\vert \bar{n}^+_{il} - \mathcal{N}^1_{il} \vert & = \vert n^+_{il}(T) - n^-_{il}(T) -\{ (\mathcal{N}^1_{+,il} - \mathcal{N}^1_{+,il} ) \}_{\mathcal{O}_7} \vert \\ 
      & \leq \vert n^l_{ip}(T) - \{\mathcal{N}^1_{+,il}\}_{\mathcal{O}_7}  \vert + \vert n^l_{in}(T)-\{\mathcal{N}^1_{-,il}\}_{\mathcal{O}_7} \vert \\
      & \leq (m^+_{4}+m^-_{4})e^{-T}  \triangleq m_4 e^{-T}.
\end{split}
\end{gather*}  
Besides, accumulative errors when zeroing species $\{N^-_{il}\}^2_{i=1}$ are deduced through revising the feasible ranges of exponential stability parameters in Lemma 1 \cite{fan2023automatic} by considering the inaccuracies in $n^{\pm}_{il}(T)$, that is
\begin{equation*}
\vert n^-_{il}(2T) - 0 \vert = \vert n^-_{il}(2T) - \bar{n}^-_{il} \vert < m^{'i}_1 e^{-\alpha^{'i}_1 T},
\end{equation*}
where $\forall m^{'i}_1 > \{\mathcal{N}^1_{-,il}\}_{\mathcal{O}_{7}} + m^-_{4}e^{-T}, \forall \alpha^{'i}_1 \in (0,\vert \mathcal{N}^1_{il} \vert - m_4e^{-T})$. Therefore, with the same revised parameters, we find the total error caused by $\{N^+_{il}\}^2_{i=1}$ ends up in 
\begin{equation*}
\vert n^+_{il}(2T) - \vert \mathcal{N}^1_{il} \vert \vert < m^{'i}_1 e^{- \alpha^{'i}_1 T} + m_4 e^{-T}.
\end{equation*}
Following the similar analysis, if $\mathcal{N}^1_{il} < 0$, we have the total error estimation with $\forall m^{'i}_2 > \{\mathcal{N}^1_{+,il}\}_{\mathcal{O}_{7}} + m^+_{4}e^{-T}, \forall \alpha^{'i}_2 \in (0,\vert \mathcal{N}^1_{il}\vert - m_4e^{-T})$ as follows
\begin{equation*}
    \begin{split}
 \vert n^+_{il}(2T) \vert & < m^{'i}_2 e^{-\alpha^{'i}_2 T},\\
 \vert n^-_{il}(2T) - \vert \mathcal{N}^1_{il} \vert \vert & < m_4 e^{-T} + m^{'i}_2 e^{-\alpha^{'i}_2 T}.
    \end{split}
\end{equation*}
Then, the unified form is obtained. 
\end{proof}

In effect, the estimation above is subject to a hidden assumption. It remains vaild when the propagation errors on initial conditions from $\mathcal{O}_7$ are unable to alter the attractive region of the annihilation reaction in $\mathcal{M}^f_{n^1_1}$. However, intuitively, another case cannot be dismissed if the error significantly impacts the attractive region. Thus, we present the following lemma to show that this scenario relies on $T$, and an appropriate phase length is feasible to identify to mitigate the worst case.

\begin{lemma}
\label{le:feasible_T}
Consider these two alternating operating systems
\begin{equation*}
\begin{array}{cc}
\sum_{1}: \begin{split}
    \dot{x} (t) &= c^m_1 -x(t) \\
    \dot{y}(t) & = c^m_2 - y(t),
\end{split}     
    ~& 
\sum_{2}:
\begin{split}
    \dot{x}(t) &=-x(t)y(t) \\
    \dot{y}(t) & =-x(t) y(t),
\end{split}
\end{array}
\end{equation*}
where constants $c^m_1, c^m_2$ will change over the  round $m$. Systems $\sum_1, \sum_2$ can separately last for any fixed time length $T$ in every round. Given the initial condition $x^1_0=y^1_0=0$ in the first round, if $T > \ln{(1+\frac{\sum ^{m-1}_{j=1}\vert c^{m-j}_1 -c^{m-j}_2\vert}{\vert c^{m}_1 -c^{m}_2\vert})}$, it holds that $(c^{m}_1-c^{m}_2)(x(2mT)-y(2mT)) \geq 0$  for all $m$.
\end{lemma}
\begin{proof}
We specify that $x^m(t)=x(t+2(m-1)T)$ when $t \in \left[0, 2T\right]$. The inequality is always true when $T= \infty$ (reach the equilibrium). When $T$ is finite, the trajectory difference between $x^m(t)$ and $y^m(t)$ for $t \in \left[0, T\right]$ is 
\begin{equation*}
x^m(t) - y^m(t) = (c^m_1-c^m_2)(1-e^{-t})+ e^{-t}(x^m_0-y^m_0).
\end{equation*}
Then, $\sum_2$ ensures that with $t \in \left(T,2T\right], x^m(t) - y^m(t) = x^m(T)-y^m(T)$. It obvious that $(c^{1}_1-c^{1}_2)(x^1(2T)-y^1(2T)) \geq 0$. For $\forall m \geq 2$, if $(c^{m}_1-c^{m}_2)(x^m_0-y^m_0) \geq 0$, the result is naturally true since the response curve of $x(t)$ is always on top of $y(t)$ in $\left[2(m-1)T,2mT\right]$. Uncertainties always happen when $(c^{m}_1-c^{m}_2)(x^m_0-y^m_0) < 0$. If $c^m_1 - c^m_2 <0$, while $x^m_0-y^m_0 = x^{m-1}(2T) - y^{m-1}(2T) > 0$, the inequality that we expect 
\begin{equation}
(c^m_1-c^m_2)(1-e^{-T})+ e^{-T}(x^m_0-y^m_0) < 0
\label{eq:ineq_lemma4}
\end{equation}
is not true for some finite $T$. Since $x^m_0-y^m_0$ can be written as the recurrence formula below
\begin{equation*}
\begin{split}
x^m_0-y^m_0 & = (c^{m-1}_1-c^{m-1}_2)(1-e^{-T}) + e^{-T}(x^{m-1}_0-y^{m-1}_0)\\
&=\sum^{m-2}_{j=0} \{e^{-jT} (1-e^{-T}) (c^{m-j-1}_1-c^{m-j-1}_2)\}\\
& < \sum^{m-1}_{j=1} \vert c^{m-j}_1-c^{m-j}_2 \vert,
\end{split}
\end{equation*}
we can find a feasible range of $T$ to make \eqref{eq:ineq_lemma4} hold, i.e.,
\begin{equation}
    e^{T} > 1+\frac{\sum^{m-1}_{j=1} \vert c^{m-j}_1 - c^{m-j}_2 \vert}{c^m_2 - c^m_1}.
    \label{eq:ineq2_lemma4}
\end{equation}
A similar analysis also applies to $c^m_1 > c^m_2, x^m_0 < y^m_0$, so the general results is denoted as $T > \ln{(1+\frac{\sum ^{m-1}_{j=1}\vert c^{m-j}_1 -c^{m-j}_2\vert}{\vert c^{m}_1 -c^{m}_2\vert})}$.
\end{proof}
\begin{remark}
It is possible to provide a rough range estimation for $T$ based solely on the information of FCNN since $c^m_1 = (w^+_{1}(0)s_a(T))_{il}$, $c^m_2=(w^-_{1}(0)s_a(T))_{il}$ in BFCNN, which are close to the standard values $\{\mathcal{N}^1_{+,il}\}_{\mathcal{O}_7}, \{\mathcal{N}^1_{-,il}\}_{\mathcal{O}_7}$. In general, the right side of \eqref{eq:ineq2_lemma4} simply requires $T$ to surpass 15-20 while considering the weights range and typical iteration counts, which is not a tough requirement for the oscillator.
\end{remark}

We now address the errors in $\mathcal{M}^f_{n^2_1}$ with equations 
\begin{equation*}
\mathcal{M}^f_{n^2_1}:~~\dot{p}^{\pm}_{il}(t) = n^{\pm}_{il} (t) (\frac{1}{2}-p^{\pm}_{il}(t))  
\end{equation*}
with $i=1,2$, under the condition that $T$ is designed within the valid range above, and the discussion is twofold. In the case of $\mathcal{N}^1_{il} > 0$, we present the distance between the trajectory and the standard value $\frac{1}{2}$ as $ \vert p^+_{il}(t) - \frac{1}{2}\vert = \vert p^+_{il} (0) - \frac{1}{2} \vert e^{-n^+_{il}(2T)t}$,    
from which the current error and propagation error stored in $n^+_{il}(2T)$, as the constant parameter here, hinders the hidden-layer output species $P^+_{il}$ from achieving the target value at the original rate, i.e., errors emerging in the power item impact the achievable accuracy of systems within a finite time. Similar arguments can be applied in the negative case.
As for $\mathcal{M}^f_{n^3_1}$, recall that the dynamical equations is designed as \cite{fan2023automatic} 
\begin{align*}
\mathcal{M}^f_{n^3_1}:~\dot{p_{il}^+}(t) &= n_{il}^+(t)p^+_{il}(t)(1-p^+_{il}(t)), ~\dot{n_{il}^+}(t) = -n^+_{il}(t),\\ 
\dot{p_{il}^-}(t) & = n_{il}^-(t)p^-_{il}(t)(p^-_{il}(t)-1),~\dot{n_{il}^-}(t) = -n^-_{il}(t),\\
\dot{p}_{il}(t) & = p^+_{il}(t) + p^-_{il}(t) - p_{il}(t).  
\end{align*}
The current error needs to be assigned the coefficient of a new valid range, and the accumulative errors on LSSs strongly depend on $p^{\pm}_{il}(T)$ and $n^{\pm}_{il}(2T)$. Then, incorporating Lemma \ref{prop_O9}, we give the total error upper bound to compute hidden-layer outputs $\Upsilon^1$ below.
\begin{proposition} \label{prop-psp}
Let $\Upsilon^1_{il}$ be the positive standard value of hidden-layer outputs, and $s_1, s_2$ be sign markers. There exist $m^{\pm}_{5i}, m^{\pm}_{6i}, \tilde{m}^{\pm}_{il}, a^{\pm}_{qi}>0$ with $i \in \{1,2\}, s \in \{s_1,s_2\}, q =\{4,5,6\}$ such that
\begin{equation*}
    \begin{split}
\vert p^{s_1}_{il}(2T) - \Upsilon^1_{il}  \vert & < (\tilde{m}^{s_1}_{il} +\frac{m_4}{4}) e^{-T} + \frac{m^{'i}_j}{4} e^{-\alpha^{'i}_j T} \\ 
 + m^{s_1}_{6i}m^{s_1}_{5i}&(m^{'i}_1 Te^{-\alpha^{'i}_j T} + m_4 Te^ {-T}+ e^{-\vert \mathcal{N}^1_{il} \vert T}),\\
 \vert  p^{s_2}_{il}(2T)  - 0 \vert &  \leq \frac{m^{'i}_j}{4} e^{-\alpha^{'i}_j T} + m^{s_2}_{5i}m^{'i}_j Te^{-\alpha^{'i}_j T},\\
\vert  p_{il}(T)  - \Upsilon^1_{il} \vert & < a^{s}_{4i}T e^{-T} +a^{s}_{5i}Te^{-\alpha^{'i}_{j} T} + a^{s}_{6i} e^{-\vert \mathcal{N}^1_{il} \vert T},
    \end{split}
\end{equation*}
where $i \in \{1,2\}$, and if $\mathcal{N}^1_{il}>0$ , $s_1 = +, s_2=-, j=1, s=s_1$, otherwise $s_1 = -, s_2 =+, j=2, s = s_2$.
\end{proposition}
\begin{proof}
Considering $\mathcal{N}^1_{il}>0$ , in $\mathcal{M}^f_{n^2_2}$, we regard the ideal trajectory of $p^+_{il}(t)$ as $p^{+\ast}_{il}(t)$ ($p^{+\ast}_{il}(0)=p^+_{il}(0)$ when $m=1$) and write the transient error estimation as follows
\begin{gather*}
    \begin{split}
\vert p^+_{il}(t) - p^{+\ast}_{il}(t)
\vert & = \vert p^+_{il}(0) - \frac{1}{2} \vert \vert e^{-n^+_{il}(2T)t} - e^{-\vert \mathcal{N}^1_{il} \vert t} \vert\\
& < \vert p^+_{il}(0) - \frac{1}{2} \vert \vert n^+_{il}(2T)t - \vert \mathcal{N}^1_{il} \vert t  \vert.
    \end{split}
\end{gather*}
It works because $e^{-x}$ has $1$ as its global lipschitz constant on $\mathbb{R}_{\geq 0}$. Substituting the error of $n^+_{il}(2T)$ obtained in the upstream module $\mathcal{M}^f_{n^2_1}$, we have 
\begin{equation*}
\vert p^+_{il}(T) - p^{+\ast}_{il}(T)
\vert < \vert p^+_{il}(0) - \frac{1}{2} \vert (m^{'i}_1 T e^{-\alpha^{'i}_1 T} + m_4 Te^ {-T}).
\end{equation*}
Then, combining $\vert p^{+\ast}_{il}(T) -p^{+}_{il} \vert = \vert p^{+\ast}_{il}(T) - \bar{p}^{+\ast}_{il} \vert = \vert 
p^+_{il}(0) -\frac{1}{2} \vert e^{-\vert \mathcal{N}^1_{il} \vert T}$, the total error ends up being 
\begin{equation*}
\vert p^+_{il}(T) - \frac{1}{2}\vert < m^{+}_{5i} (m^{'i}_1 T e^{-\alpha^{'i}_1 T} + m_4 Te^ {-T}) +  m^+_{5i} e^{-\vert \mathcal{N}^1_{il} \vert T},
\end{equation*}
where $m^+_{5i} = \vert p^+_{il}(0) -\frac{1}{2} \vert$. Regarding the negative species $P^-_{il}$, we have $\vert p^-_{il}(T) - 0 \vert = \vert p^-_{il}(T) - p^-_{il}(0) \vert < m^-_{5i} m^{'i}_1 T e^{-\alpha^{'i}_1 T}$, where $m^-_{5i} = \vert p^-_{il}(0) -\frac{1}{2} \vert$. A similar argument shows that, in the case of $\mathcal{N}^1_{il} < 0$, the error upper bounds of two species concentration are switched, i.e., $\vert p^+_{il}(T) - 0 \vert < m^+_{5i}m^{'i}_2Te^{-\alpha^{'i}_2 T}, \vert p^-_{il}(T) - \frac{1}{2} \vert < m^{-}_{5i} (m^{'i}_2 T e^{-\alpha^{'i}_2 T} + m_4 Te^ {-T}) +  m^{-}_{5i} e^{-n^-_{il}T}$. 

In phase $\mathcal{O}_{13}$, Proposition 2 in \cite{fan2023automatic} gives the current error of $\{P^{\pm}_{il}\}$, that is, $\vert p^+_{il}(t) - \bar{p}^+_{il} \vert < M^{il}_{p} e ^{-t}, \vert p^-_{il}(t) - \bar{p}^-_{il} \vert < M^{il}_{n} e ^{-t}$. The initial stage of accumulative error involves adjusting $M^{il}_p, M^{il}_n$ and the new parameters having the updated range
\begin{gather*}
\tilde{m}^{+}_{il} > (\vert \mathcal{N}^1_{il} \vert +m_4e^{-T})\times I_1  + m_{j}^{'i}e^{-\alpha^{'i}T},\\
\tilde{m}^{-}_{il} > (\vert \mathcal{N}^1_{il} \vert +m_4e^{-T})\times I_2  + m_{j}^{'i}e^{-\alpha^{'i}T},
\end{gather*}
where if $\mathcal{N}^1_{il} > 0$, $I_1=1, I_2 =0, j=1$, otherwise, $I_1=0, I_2=1, j=2$. To derive the accumulative error for LSSs when $\mathcal{N}^1_{il} > 0$, we make the binary continuous function $h_p(p^+_{il}(T),n^+_{il}(2T)) $ represent $\bar{p}^+_{il}$ in $\mathcal{M}^f_{n^3_1}$ and use $(p^{+}_{il_0}, n^{+}_{il_0})$ to denote the standard value in $\mathcal{M}^f_{n^2_1}$. Utilizing the differential mean value theorem, we have an estimation
\begin{gather}
\label{eq:prop2_5}
    \begin{split}
 & \vert h_p(p^+_{il}(T), n^+_{il}(2T)) -  h_p(p^{+}_{il_0}, n^{+}_{il_0})   \vert \leq \vert h_p(p^+_{il}(T), n^+_{il}(2T)) \\
 & - h_p(p^+_{il}(T),n^{+}_{il_0}) \vert 
 + \vert h_p(p^+_{il}(T), n^{+}_{il_0}) -h_p(p^{+}_{il_0}, n^{+}_{il_0}) \vert\\
 & \leq \vert \frac{\partial h_p}{\partial n^+_{il}(2T)}  \vert \vert n^+_{il}(2T) - n^{+}_{il_0} \vert + \vert \frac{\partial h_p}{\partial p^+_{il}(T)} \vert \vert p^+_{il}(T) - p^{+}_{il_0} \vert.
    \end{split}
\end{gather}
Since
\begin{gather*}
    \begin{split}
\vert \frac{\partial h_p}{\partial n^+_{il}(2T)} \vert \leq \frac{1}{4}, \vert \frac{\partial h_p}{\partial p^+_{il}(T)} \vert \leq \frac{1}{\vert 4 p^+_{il}(T) (1-p^+_{il}(T)) \vert} \triangleq m^{+}_{6i}
    \end{split}
\end{gather*} 
for $\forall (p^+_{il}(T), n^+_{il}(2T))$, we have the total error upper bound  
\begin{gather*}
    \begin{split}
\vert  p^+_{il}(2T) &- \Upsilon^1_{il}  \vert < (\tilde{m}^{+}_{il} +\frac{m_4}{4}) e^{-T} + \frac{m^{'i}_1}{4} e^{-\alpha^{'i}_1T} \\ & + m^+_{6i}m^+_{5i}(m^{'i}_1 Te^{-\alpha^{'i}_1 T} + m_4 Te^ {-T}+ e^{-\vert \mathcal{N}^1_{il} \vert T}).
    \end{split}    
\end{gather*}
Furthermore, from the analysis above, for $\forall \delta \in (0, \frac{1}{3})$, $\exists$ $T>0$ large enough such that $p^+_{il}(T) \in B(\frac{1}{2}, \delta)$, leading to $ m^+_{6i} \in B(1, \frac{4 \delta^2}{1-4 \delta^2})$ near one. Then, the total error estimation of species $N^+_{il}$ is $\vert n^+_{il}(3T) - 0 \vert = \vert n^+_{il}(2T) \vert e^{-T} < m^{'i}_1 e^{(-\alpha^{'i}_{1}-1) T} + m_4 e^{-2T} + \vert \mathcal{N}^1_{il} \vert e^{-T}$. Actually, reactions involving species $\{N^-_{il}, P^-_{il}\}$ occur at a very slow speed because the accumulative error causes a nonzero initial concentration of catalyst $N^-_{il}$ in $\mathcal{M}^f_{n^3_1}$. Hence, the distance from $p^-_{il}(t)$ to the ideal trajectory is measured as follows
\begin{gather*}
\begin{split}
 \vert p^-_{il}(2T) -p^{-\ast}_{il}(2T) \vert  \leq \vert \frac{\partial \phi}{\partial n^-_{il}(2T)} \vert \vert n^-_{il}(2T)\vert \leq \frac{m^{'i}_1}{4} e^{-\alpha^{'i}_1 T}, 
\end{split}
\end{gather*}
where $\phi \triangleq \phi (p^-_{il}(T), n^-_{il}(2T), t)$ denotes the flow of the reaction subsystem composed of $\{N^-_{il}, P^-_{il}\}$ in $\mathcal{M}^f_{n^3_1}$ from $(p^-_{il}(T), n^-_{il}(2T))$. Since $\vert p^{-\ast}_{il}(2T) - 0 \vert \leq \vert p^{-\ast}_{il}(2T) - p^-_{il}(T)\vert + \vert p^-_{il}(T) - p^-_{il}(0) \vert < 0 + m^-_{5i}m^{'i}_1Te^{-\alpha^{'i}_1 T} $, where the zero item owing to $n^{-\ast}_{il}(2T) = 0$, the total error becomes 
\begin{equation}
\label{eq:prop2_8}
 \vert p^-_{il}(2T) - 0 \vert \leq \frac{m^{'i}_1}{4} e^{-\alpha^{'i}_1 T} + m^-_{5i}m^{'i}_1Te^{-\alpha^{'i}_1 T}.
\end{equation} 
In addition, $\vert n^-_{il}(3T) - 0 \vert < m^{'i}_1 e ^{-(\alpha^{'i}_1+1) T}$. In this way, we can make a similar but switched conclusion on error upper bounds of positive and negative species when $\mathcal{N}^1_{il} < 0$. 

To deduce the final deviation on species $P^l_i$ designed for obtaining $\Upsilon^1_{il}$, we found that only the current error upper bound of $P^+_{il}$ depends on transient time  while $T$ determines the other accumulative items, then if $\mathcal{N}^1_{il} > 0$ one can have
\begin{gather*}
\begin{split}
\vert  p_{il}(\tilde{t}) - \Upsilon^1_{il} \vert &  \leq e^{-\tilde{t}} ( \int_{0}^{\tilde{t}} \vert p^+_{il}(\tau)-\bar{p}^+_{il} \vert e^{\tau} \mathrm{d}\tau + \int_{0}^{\tilde{t}} \vert p^-_{il}(\tau) \vert e^{\tau} \mathrm{d}\tau ) \\ 
& ~~~ + e^{-\tilde{t}} \vert p^l_i(0)- \bar{p}^+_{il} \vert + \vert \bar{p}^+_{il} - \Upsilon^1_{il}  \vert.
\end{split}
\end{gather*}
Here, $\tilde{t} \in \left[0,T\right]$ and the initial values of $p^{\pm}_{il}(\tilde{t})$ refers to the terminal value in $\mathcal{M}^f_{n^2_1}$ with a slight abuse of notations. Then, substituting the current error, \eqref{eq:prop2_5} and \eqref{eq:prop2_8}, and retaining the leading items, 
we find an appropriate parameter set $\{a^s_{qi}\}^6_{q=4}$ such that the error estimation become 
\begin{gather*}
\label{eq:prop2_10}
\begin{split}
\vert p_{il}(T) -\Upsilon^1_{il} \vert  < a^s_{4i}T e^{-T} +a^s_{5i}Te^{-\alpha^{'i}_{j} T} + a^s_{6i} e^{-\vert \mathcal{N}^1_{il} \vert T}
\end{split}
\end{gather*} 
with $i=1,2$. Note that the estimation above integrates two cases of the error upper bounds. Specifically, $s = +, j=1$ if $\mathcal{N}^1_{il} > 0$ and $s = -, j=2$ in another case. 
\end{proof}

\subsubsection{Error in the Output Layer} For the second layer in BFCNN and $ l \in \{1,\cdots,\tilde{p}\}$, we utilize $\tilde{n}^{\pm}(t) \in \mathbb{R}^{1 \times \tilde{p}}_{\geq 0}$, $y(t)$, $y^{\pm}(t) \in \mathbb{R}^{1 \times \tilde{p}}_{\geq 0}$, to denote concentration vectors of species $\{N^{\pm}_{3l}\}$, global output species $\{Y^l\}$, and output species pairs $\{Y^{\pm}_l\}$, respectively, and introduce the matrix variable $p_a(t) = ((p_1(t))^{\top}, (p_2(t))^{\top}, \mathds{1})^{\top}$ for the subsequent convenience. The unified structure of the feedforward part causes errors in every layer possessing the same composition that can be derived following the workflow as before. Here, we display the error estimation for the output layer without illustration.
\begin{corollary}\label{coro_1}
Suppose that $\tilde{\mathcal{N}}^1_+, \tilde{\mathcal{N}}^1_-$ represent the standard positive and negative part of $\tilde{\mathcal{N}}^1$, and $s$ pertains to the sign marker. Then, for $\forall l \in \{1,\cdots,\tilde{p}\}$, there exist $m^{\pm}_{7}, a_{q}, \tilde{e}_{\tilde{q}}>0$ with $q \in \{4,5,6\}$ and $\tilde{q} \in \{1,...,5\}$ and powers $\alpha, \alpha_3, n^1 > 0$ for the following estimation
\begin{gather*}
    \begin{split}
\Vert \tilde{n}^{\pm}(T) - \tilde{\mathcal{N}}^1_{\pm} \Vert & \leq  m^{\pm}_{7} (a_4 Te^{-T} + a_5Te^{-\alpha T} + a_6 e^ {-n^1 T}), \\
\vert y_l(T)-y^1_l \vert & < \tilde{e}_1 Te^{-\alpha^l_3 T} + \tilde{e}_2 T^2e^{-T} + \tilde{e}_3 T^2 e^{-\alpha T} \\
& ~~~ + \tilde{e}_4 Te^{-n^1 T} + \tilde{e}_5e^{-\vert \tilde{\mathcal{N}^1_{l}} \vert T},
    \end{split}
\end{gather*}
where $ a_4 = \mathop{\rm{max}}_{s,i,l} \{ a^s_{4i}\}+\mathop{\rm{max}}_{s} \{\frac{\Vert w^{s}_{2}(0)\{p_a\}_{\mathcal{O}_{13}} - \tilde{n}^{s}(0) \Vert}{m^s_{7}}\}$, $m^{\pm}_{7} = 2 \Vert w^{\pm}_{2}(0) \Vert$, $\{a_5, \alpha\} = \mathop{\arg\max}\limits_{s,i,j,l}\{a^{s}_{5i}Te^{-\alpha^{'i}_j T}\}$, $\{a_6, n^1\} = \mathop{\arg\max}\limits_{s,i,l} \{a^s_{6i}Te^{-\vert \mathcal{N}^1_{il} \vert T}\}$.
\end{corollary}
\begin{proof}
(Sketch) The dynamical equations are omitted here. The case of linear computation is trivial. Following Proposition \ref{prop-psp}, for the sigBCRN in $\mathcal{M}^f_{n^j_2}$ and $s_1, s_2$ being sign markers, there exist the corresponding parameters $\tilde{m}^{\pm}_{y}$, $m^{\pm}_3$, $\alpha^{\pm}_3$,$m^{\pm}_8$,$m^{\pm}_9 >0$ such that in phase $\mathcal{O}_{21}$, we obtain 
\begin{gather*}
    \begin{split}
\vert & y^{s_1}_{l}(2T) - y^1_l \vert  \leq \tilde{m}^{s_1}_{y} e^{-T} + \frac{1}{4}(m^{s_1}_3 e^{-\alpha^{s_1}_3 T} + m_7 \{A\}) \\
& + m^{s_1}_{9}m^{s_1}_{8}(m^{s_1}_3 T e^{-\alpha^{s_1}_3 T} + m_7\{A\}T + e^{-\vert \tilde{\mathcal{N}}^1_{l} \vert T}),\\
\vert & y^{s_2}_{l}(2T) - 0 \vert \leq \frac{1}{4} m^{s_1}_3e^{-\alpha^{s_1}_3 T} + m^{s_2}_{8} m^{s_1}_3 T e^{-\alpha^{s_1}_3 T},
    \end{split}
\end{gather*}
where $m_7 = m^{+}_7 + m^{-}_7$, $\{A\}=a_4Te^{-T} + a_5Te^{-\alpha T} + a_6e^{-\tilde{n}^1 T}$, and if $\tilde{\mathcal{N}}^1_{l}>0$, $s_1 = +, s_2 =-$, otherwise $s_1 = -, s_2 =+$. Then, we can have the final upper bound.
\end{proof}

Hereafter, we denote the error upper bound of computing hidden-layer outputs $p^1_{il}$ and global outputs $y^1_l$ as $\{A\}$ and $\{E\}$, respectively, for ease of subsequent expression. 
\subsection{Preceding Calculation Module Error}
The pre-calculation BCRN (pBCRN), solely occupying $\mathcal{O}_{23}$, comprises boundary equilibrium reactions $\mathcal{O}^1_{23}$ and type-\Rmnum{1} CRNs $\mathcal{O}^2_{23}$, $\mathcal{O}^3_{23}$ \cite{fan2023automatic}. Recall that for $i=\{1,2\}$, $y_e(t), y_s(t), \tilde{y}(t), i_y(t) \in \mathbb{R}^{1 \times \tilde{p}}$ and $p_{is}(t), \tilde{p}_i(t), i_{p_i}(t) \in \mathbb{R}^{1 \times \tilde{p}}$ indicates concentration vectors of
$\{Y^l_e, Y^l_s, \tilde{Y}^l, I^l_y\}$ of MAS $\mathscr{M}^1$, and $\{P^l_{is},\tilde{P}^l_{i}, I^l_{p_i}\}$ of $\mathscr{M}^2_{i}$, respectively, in $\mathcal{O}^1_{23}$ with equations 
\begin{gather*}
\begin{split}
\mathscr{M}^1:~\dot{y}_l(t) & = - k y_l(t), ~\dot{\tilde{y}}_l(t) =  k y_l(t),\\
    \dot{y}^l_e(t) &= k y_l(t) - k y^l_e(t)s^l_3(t),~\dot{s}^l_3(t)=-k y^l_e(t) s^l_3(t),\\
    \dot{y}^l_{s}(t) &= k y_l(t)-ky^l_s(t)i^l_y(t),~\dot{i}^l_y(t)=-k y^l_s(t)i^l_y(t).~~\\
\mathscr{M}^2_{i}:~\dot{p}_{il}(t) & = - k p_{il}(t), ~ \dot{p}^l_{is}(t) = k p_{il}(t) - k p^l_{is}(t)i^l_{p_i}(t), \\
~ \dot{\tilde{p}}_{il}(t) &= k p_{il}(t),~\dot{i}^l_{p_i}(t)=-k p^l_{is}(t)i^l_{p_i}(t).~~
\end{split}
\end{gather*}
Then, $e^{\pm}(t)$, $s^l_{y}(t)$, $s^l_{p_i}(t)$, $e(t) \in \mathbb{R}^{1 \times \tilde{p}}$ separately denote the concentration vectors of species $\{E^{\pm}_l\}$, $\{S^l_y\}$, $\{S^l_{p_i}\}$, $\{E_l\}$ in $\mathcal{O}^2_{23}$, $\mathcal{O}^3_{23}$ with dynamical equations according to type-\Rmnum{1} CRN 
\begin{gather*}
\begin{split}
\mathcal{O}^2_{23}:~\dot{e}^{+}_l(t) & = s^l_3(t)-e^{+}_l(t), ~\dot{e}^{-}_l(t) = y^l_e(t)-e^{-}_l(t),\\
\dot{s}^{l}_y(t) & = i^l_y(t)-s^l_y(t), ~\dot{s}^{l}_{p_i}(t) = i^l_{p_i}(t)-s^l_{p_i}(t),\\
\mathcal{O}^3_{23}:~\dot{e}_l(t)&=e^+_l(t)+e^-_l(t)-e_l(t).
\end{split}
\end{gather*}
Part \Rmnum{1} elaborates that there exist $M_x, M_y > 0$ and exponential convergence orders $K, \tilde{K}>0$ for any fixed $T$ of $\mathscr{M}^1, \mathscr{M}^2$ with boundary LSSs, which
facilitate to estimate the current error of $\mathcal{O}^1_{23}$. For the other two, following the derivation in Lemma 2 of Part \Rmnum{1}, the concrete current errors are combined into two scenarios, i.e., if $\tilde{K}, K^i \neq 1$, we have 
\begin{equation}
\begin{split}
\Vert x^l_e(T) - \bar{x}^l_e \Vert & < n_x(e^{-\tilde{K}T}-e^{-T}) + g_xe^{-T}, \\
\vert s^l_{p_i}(T) - \bar{s}^l_{p_i} \vert & < n^i_y(e^{-K^iT} - e^{-T}) + g_{p_i}e^{-T}, \\
\vert e_l(T) - \bar{e}_l \vert  < 2\tilde{n}_x e^{-\tilde{K}T} & +(g^+ + g^- -2n_x ) Te^{-T} + g_e e^{-T},
\end{split}
\label{eq:pre_1}
\end{equation}
where $x^l_e(t)=(e^+_l(t),e^-_l(t),s^l_y(t))^{\top}$, $n_x = \frac{M_x}{1-\tilde{K}}$, $n^i_y = \frac{M^i_y}{1-K^i}$, $\tilde{n}_x = \frac{M_x}{(1-\tilde{K})^2}$, $g^{\pm} = \vert e^{\pm}_l(0) - \bar{e}^{\pm}_l\vert$, $g_{p_i}=\vert s^l_{pi}(0) - \bar{s}^l_{pi} \vert$, $g_e = \vert e_l(0) - \bar{e}_l \vert$, $g_x = \max\{g^{\pm}, \vert s^l_y(0) - \bar{s}^l_y \vert \}$ and if $\tilde{K}, K^1, K^2 = 1$, they become 
\begin{equation}
    \begin{split}
\Vert x^l_e(T) - \bar{x}^l_e \Vert & < M_x T e^{-T} + g_x e^{-T}, \\
\vert s^l_{p_i}(T) - \bar{s}^l_{p_i} \vert & < M^i_y Te^{-T} + g_{p_i} e^{-T},\\
\vert e_l(T) -\bar{e}_l \vert < M_x T^2e^{-T}& + (g^+ + g^-) Te^{-T}+ g_e e^{-T}.
    \end{split}
\label{eq:pre_2}
\end{equation}
Now, we discuss the total error bound of every type of species. The error propagation schematic is shown in Fig.\ref{fig:7}.
\begin{figure}[!t]
\centerline{\includegraphics[width=\columnwidth]{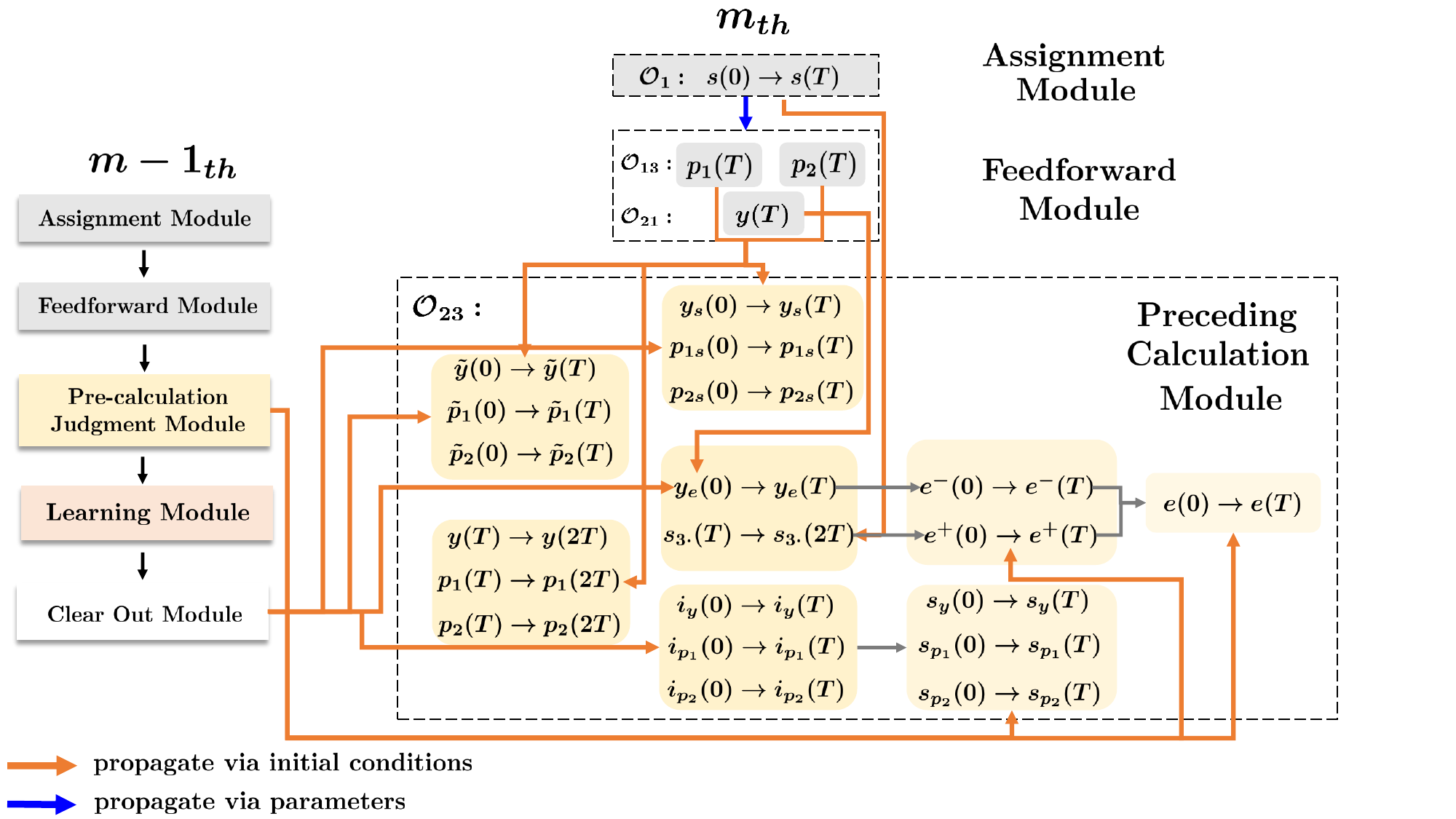}}
\caption{Schema of the propagation of error via initial concentrations and parameters from the $m-1$th to the $m$th iteration in the Preceding Calculations module.}
\label{fig:7}
\end{figure}
Firstly, since no errors accumulated in the LSSs of species $\{Y^l, P^l_i\}$ participating in pBCRN, and the analytic solutions are available, we have
\begin{equation*}
    \begin{split}
\vert y_l(2T) - 0 \vert & =\vert y_l(T)\vert e^{-kT} < y^1_l e^{-kT} + \{E\} e^{-kT},\\
\vert p_{il}(2T) - 0 \vert & =\vert p_{il}(T)\vert e^{-kT} < \Upsilon^1_{il} e^{-kT} + \{A\} e^{-kT},
    \end{split}
\end{equation*}
where $k$ is the rate constant in $\mathcal{O}^1_{23}$. Then, in our design, the equilibrium concentration of species $\{\tilde{Y}^l, \tilde{P}^l\}$ are used to replace that of the original output species $\{Y^l, P^l_i\}$ for further update computation. Thus, their deviation from the standard value is given closer attention below. 
\begin{lemma}
Suppose that the rate constant is taken as $k > \mathop{\rm{max}}\{1, \tilde{n}^1, \alpha, \vert \mathcal{N}^1_l \vert, \alpha_3\}$, the errors in obtaining $ y^1_l$, $\Upsilon^1_{il}$ caused by $\{\tilde{Y}^l, \tilde{P}^l_i\}$ are of the same order as those resulting from species $\{Y^l, P^l_i\}$. 
\end{lemma}
\begin{proof}
Their current errors are equal to $\vert y_l(T) \vert e^{-kT}$ and $\vert p^l_i(T) \vert e^{-kT}$, respectively. Fig.\ref{fig:7} displays that the accumulative error of $\tilde{Y}^l$ on LSSs is caused by the initial concentration $\tilde{y}_l(0)$ and the vertical propagation from $y_l(T)$, i.e.,
\begin{equation*}
\vert \bar{\tilde{y}}_l - y^1_l \vert \leq \vert \bar{\tilde{y}}_l - y_l(T) \vert + \vert y_l(T) - y^1_l \vert = \vert \tilde{y}_l(0) \vert + \vert y_l(T) - y^1_l \vert.
\end{equation*}
Note that $m=1$, so only the latter part needs to be considered, and the total error upper bound ends up being
\begin{equation*}
\vert \tilde{y}_l(T) - y^1_l \vert \leq y^1_l e^{-kT} + \{E\} e^{-kT} + \{E\}. 
\end{equation*}
Likewise, we have the error estimation of species $\tilde{P}^l_i$ 
\begin{equation*}
\vert \tilde{p}_{il}(T) -\Upsilon^1_{il} \vert \leq \Upsilon^1_{il} e^{-kT} + \{A\} e^{-kT} + \{A\}.
\end{equation*}   
Furthermore, all entries can be consolidated into $\{E\}, \{A\}$ since $k$ is greater than the powers within them.
\end{proof}

As for other species in $\mathcal{O}^1_{23}$, the estimates are as follows.
\begin{lemma} \label{lem_bound2}
For $l \in \{1,\cdots,\tilde{p}\}$ and $i=\{1,2\}$, let $s^l_3, y^l_e$ denote standard values computed by species $\{S^l_3, Y^l_e\}$ in $\mathcal{O}^1_{23}$, then the error upper bounds of species $\{S^l_3, Y^l_e\}$ and $\{Y^l_s, I^l_{y}\}$, $\{P^l_{is}, I^l_{p_i}\}$ become
\begin{gather*}
\begin{split}
\vert y^l_s (T) -0\vert  < M_x e^{-\tilde{K}T} &, ~ \vert p^l_{is} (T) -0\vert < M^i_y e^{-K^i T},\\
\vert i^l_y(T) - (1-y^1_l) \vert & <  M_x e^{-\tilde{K}T} + \{E\},\\
\vert i^l_{p_{i}}(T) - (1-\Upsilon^1_{il}) \vert & < M^i_y e^{-K^i T} + \{A\}, \\
\vert y^l_e(T) - y^l_e \vert  <M_x e^{-\tilde{K}T} & + \mathbb{I}_{\{(y_l(T) > s^l_3(T)\}}(m_1e^{-T} + \{E\}), \\
\vert s^l_3(2T) - s^l_3 \vert  <  M_x e^{-\tilde{K}T} & + \mathbb{I}_{\{(y_l(T) < s^l_3(T)\}}(m_1e^{-T} + \{E\}),
\end{split}    
\end{gather*}
where $\mathbb{I}$ is an indicator function.
\end{lemma}
\begin{proof}
As shown in Fig.\ref{fig:7}, accumulative errors on boundary LSSs of these species come from their initial concentrations as well as vertical propagation from $y_l(T), p_{il}(T)$. For $\{Y^l_s, I^l_y\}$ and $\{P^l_{is}, I^l_{p_i}\}$, we have $y^l_s(0)=p^l_{is}(0)=0$, $i^l_y(0)=i^l_{p_i}(0) = 1$, so $y^l_s(t), p^l_{is}(t) \leq 1$ for $\forall t$ as long as $y_l(T), p_{il}(T) \leq 1$. 
In this case, we have the following formulas for error estimation
\begin{equation*}
  \begin{split}
\vert & y^l_s (T) -0\vert  < M_x e^{-\tilde{K}T}, ~ \vert p^l_{is} (T) -0\vert < M^i_y e^{-K^i T},\\
\vert & i^l_y(T) - i^l_y \vert < M_x e^{-\tilde{K}T} + \vert i^l_y(0) - (y^l_s(0)+y_l(T)) -i^l_y \vert\\
& < M_x e^{-\tilde{K}T} +  \vert i^l_y(0)-1 \vert  + \vert y^l_s(0)  \vert + \vert y_l(T) - y^1_l \vert, \\
\vert & i^l_{p_{i}}(T) - i^l_{p_i} \vert  < M^i_y e^{-K^i T} + \vert i^l_{p_i} (0) - 1 \vert + \vert p^l_{is}(0)  \vert \\
&~~~~~~~~~~~~~~~~~~ + \vert p_{il}(T) - \Upsilon^1_{il} \vert.
  \end{split}
\end{equation*}
Here, $i^l_y = 1-y^1_l$, $i^l_{p_i} = 1- \Upsilon^1_{il}$, and the error item caused by initial values can be eliminated due to $m=1$. For $\{Y^l_e, S^l_3\}$, we have $y^l_e(0)=0$, and the concentration trajectories may fall in attractive regions of different equilibrium points, which depends on $y_l(T), s^l_3(T)$. When $y_l(T) < s^l_3(T)$, 
\begin{equation*}
    \begin{split}
\vert & y^l_e(T) - y^l_e\vert = \vert y^l_e(T) - 0\vert < M_x e^{-\tilde{K}T},\\
\vert & s^l_3(2T) - s^l_3 \vert  < M_x e^{-\tilde{K}T} + \vert s^l_3(T) - (y^l_e(0) + y_l(T))  - s^l_3\vert \\
& < M_x e^{-\tilde{K}T} + \vert s^l_3(T) - d^l \vert + \vert y^l_e(0)\vert + \vert y_l(T)  -y^1_l \vert.  
    \end{split}
\label{eq:y_e-s_31}
\end{equation*}
In another case, $y^l_e=-(d^l-y^1_l)$, then the upper bound regarding two species has been switched as
\begin{equation*}
\label{eq:y_e-s_32}
    \begin{split}
\vert & y^l_e(T) - y^l_e\vert  < M_x e^{-\tilde{K}T} + \vert y^l_e(0) + y_l(T) - s^l_3(T) -y^l_e \vert \\
& < M_x e^{-\tilde{K}T} + \vert y^l_e(0) \vert + \vert s^l_3(T) - d^l \vert 
+ \vert y_l(T) - y^1_l \vert, \\
\vert & s^l_3(T) - s^l_3\vert = \vert s^l_3(T) - 0\vert < M_x e^{-\tilde{K}T}.
    \end{split}
\end{equation*}    
\end{proof}

\begin{remark}
Mild assumptions should be imposed on the fixed $T$ to avoid the radius of deviation of actual concentrations from standard values leading to $y^l_e(0)+y^1_l(T) > s^l_3(T)$ while $y^1_l < d^l $, or $y^l_e(0) + y^1_l(T) < s^l_3(T)$ while $y^1_l > d^l$, and causing $y_l(T)+y^l_s(0),~ p_{il}(T)+p^l_{is}(0) > 1$, which implies that errors alter the original attractive regions. Here, $T$ is assumed to be taken in the implicit feasible range 
\begin{equation*}
\{E\}+m_1e^{-T} < \vert d^l-y^1_l \vert, ~~ \{E\} < 1-y^1_l,
\end{equation*}
if $m=1$. It suggests that $T$ must possess an additional lower bound, similar to the one presented in Lemma \ref{le:feasible_T}. 
\end{remark}

Regarding the type-\Rmnum{1} CRN in $\mathcal{O}^2_{23}$, errors caused by $\{S^l_y, S^l_{p_i}\}$ depends on $K^1, K^2, \tilde{K}$, i.e., the upper bound is 
\begin{equation*}
    \begin{split}
\vert s^l_y(T) - (1-y^1_l) \vert & < \vert s^l_y(T) - \bar{s}^l_y \vert + \vert \bar{i}_y - i^l_y \vert \\
 & < \mathbb{I}_{\{\tilde{K} \neq 1\}}n_x (e^{-\tilde{K}T}-e^{-T})+ g_y e^{-T} \\
& ~~~ + \mathbb{I}_{\{\tilde{K} = 1\}}M_x T e^{-T} + \{E\},\\
 \vert s^l_{p_i}(T) - (1-\Upsilon^1_{il}) \vert
 & < \vert s^l_{p_i}(T) - \bar{s}^l_{p_i} \vert + \vert \bar{i}^l_{p_i} - i^l_{p_i} \vert \\
 & < \mathbb{I}_{\{K^i \neq 1\}} n^i_y (e^{-K^iT}-e^{-T})+ g_{p_i}e^{-T} \\
 & ~~~ +\mathbb{I}_{\{K^i = 1\}}M^i_y Te^{-T}  + \{A\}.
    \end{split}
\end{equation*}
Note that $\bar{s}^l_y = \bar{i}^l_y, \bar{s}^l_{p_i}=i^l_{p_i}$, $g_y = \vert s^l_y(0) -i^l_y \vert + \{E\}$. As for error species pairs $\{E^+_l, E^-_l\}$, the upper bound should be estimated in $2 \times 2$ cases, as there are distinct scenarios for current and accumulative errors. To elaborate, setting $e^+_l, e^-_l$ be the standard value and combining \eqref{eq:pre_1} and \eqref{eq:pre_2}, we derive 
\begin{equation*}
\begin{split}
\vert  e^+_l(T) -e^+_l \vert & <  \mathbb{I}_{\{\tilde{K} = 1\}}M_x T e^{-T} + \mathbb{I}_{\{\tilde{K} \neq 1\}}n_x (e^{-\tilde{K}T}-e^{-T})\\
& ~~~ + g^+ e^{-T} + \mathbb{I}_{\{y^1_l < d^l\}}(m_1 e^{-T} + \{E\}),\\
\vert  e^-_l(T) - e^-_l \vert & < \mathbb{I}_{\{\tilde{K} = 1\}}M_x T e^{-T} + \mathbb{I}_{\{\tilde{K} \neq 1\}}n_x (e^{-\tilde{K}T}-e^{-T}) \\
& ~~~ + g^- e^{-T} + \mathbb{I}_{\{y^1_l > d^l\}}(m_1 e^{-T} + \{E\}).
\end{split}    
\end{equation*}
Now, considering $\mathcal{O}^3_{23}$, standard training errors $e^l = d^l-y^1_l$ for all $l$ are finally stored in species $\{E^l\}$ and the realization error is described as
\begin{equation*}
    \begin{split}
        \vert e_l(T) - e^l\vert < \vert e_l(T) - \bar{e}_l \vert + m_1 e^{-T} + \{E\}
    \end{split}.
\end{equation*}
The upper bound holds whatever the sign of $e^l$ and
may differ in current error relying on $\tilde{K}$, as \eqref{eq:pre_1} and \eqref{eq:pre_2}. Notably,
similar to $g_y$, propagation errors from $s^l_3(T), p_{il}(T), y_l(T)$ stored on LSSs make coefficients become $g^+ = \vert e^+_l(0) - s^l_3 \vert + \mathbb{I}_{\{e^l > 0\}} \epsilon_e $, $g^- = \vert e^-_l(0) -y^l_e \vert + \mathbb{I}_{\{e^l < 0\}} \epsilon_e $, $g_{p_i}=\vert s^l_{pi}(0) - (1-\Upsilon^1_{il})) \vert + \{A\}$, and $g_e = \vert e_l(0) - e^l \vert +\epsilon_e$ with $\epsilon_e=m_1e^{-T} + \{E\}$.

Realization errors stemming from the Judgment Module are not treated in isolation within the scope of this paper. The total errors impact the reaction rates in the learning module without explicitly altering their LSSs based on the design \cite{fan2023automatic}. In essence, this influence is akin to that of a chemical oscillator and can be mitigated by adjusting subsequent reaction rate constants and meanwhile, adding clear-out reactions to manually zero some concentrations.

\subsection{Learning Module Error}
\begin{figure}[!t]
\centerline{\includegraphics[width=\columnwidth]{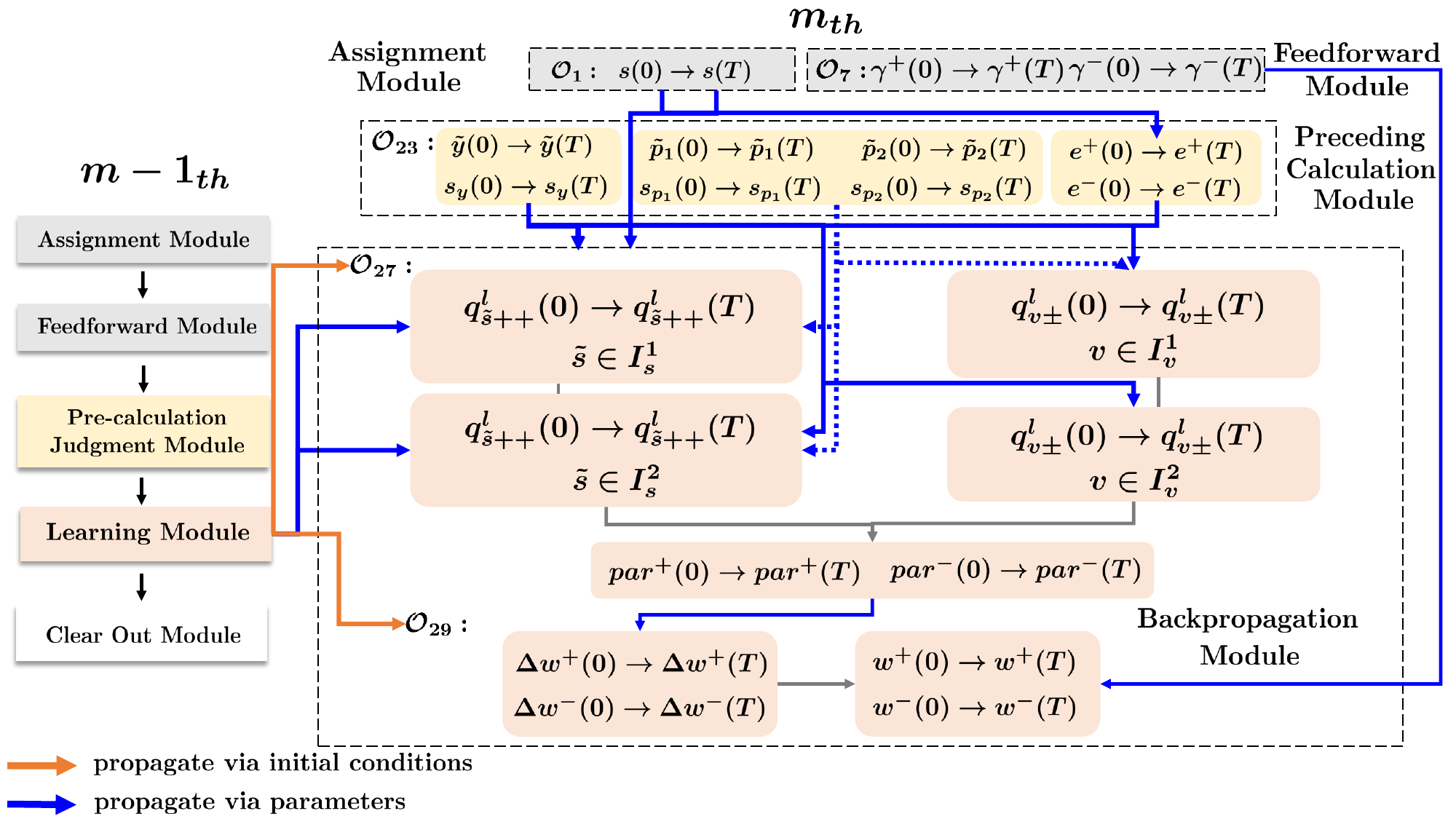}}
\caption{Schema of the propagation of error via initial concentrations and parameters from the $m-1$th to the $m$th iteration in the Learning Module.}
\label{fig:8}
\end{figure}
The learning module comprises the sub-computational module $\mathcal{M}^{le}_1$ designed for computing negative gradients and $\mathcal{M}^{le}_2$ intended for obtaining updated weights, operating during phases $\mathcal{O}_{27}$ and $\mathcal{O}_{29}$, separately. 
The current error and accumulative error in $\mathcal{M}^{le}_1$ are both segmented into the multiplication component (arising from $\mathcal{O}^1_{27}$) and the addition component (arising from $\mathcal{O}^2_{27}$). Notably, as illustrated in Fig.\ref{fig:8}, almost every factor in \eqref{eq:weight_update}
computed by the upstream BCRN modules is inaccurate, resulting in the accumulative error on LSSs of $\mathcal{O}^1_{27}$ through equation parameters. 
Now we present a general upper bound estimation for such multiplication errors.  
\begin{lemma} \label{lem_multi}
Suppose that we have $\vert \bar{x}_i - x_i \vert <\epsilon_i$ for $\forall i \in \{1,2,...,n\}$, where $x_i$ is constant, $\epsilon_i > 0$ denotes a sufficiently small upper bound, and $\bar{x}=\left[\bar{x}_1,...,\bar{x}_n\right]^{\top} \in \Omega$, a bounded set in $\mathbb{R}^n$. Then, there exists a set of $\{h_i\}^n_{i=1} > 0$ such that $\vert \prod^n_{i=1} \bar{x}_i - \prod^n_{i=1}x_i \vert < \sum^n_{i=1} h_i \epsilon_i$.
\end{lemma}
\begin{proof}
Let's say $\exists f_i > 0$ such that $\vert \bar{x}_i \vert < f_i$ since all $\bar{x}_i$ is bounded. Let $c_j = \vert \prod^{j}_{i=1} x_i \vert$, then we have
\begin{equation*}
    \begin{split}
\vert \prod^{n}_{i=1} \bar{x}_i - \prod^{n}_{i=1} x_i \vert  < f_n \vert \prod^{n-1}_{i=1} \bar{x}_i - \prod^{n-1}_{i=1} x_i \vert + c_{n-1} \vert \bar{x}_n - x_n \vert.
    \end{split}
\end{equation*} 
Using $r_n$ to denote the error derivation of $n$-factor multiplication $\vert \prod^n_{i=1} \bar{x}_i - \prod^n_{i=1}x_i \vert$, we write the recursive inequality
\begin{equation*}
    \begin{split}
r_n & < f_n R_{n-1} + c_{n-1} \epsilon_n \\
& < f_n (f_{n-1}r_{n-2} + c_{n-2} \epsilon_{n-1} ) +c_{n-1} \epsilon_n \\
& < \prod^{n}_{i=2} f_i r_1 + \sum^{n}_{i=2} (c_{i-1} \epsilon_i \prod^{n}_{k=i+1} f_k) \\
& < \prod^{n}_{i=2} f_i \epsilon_1 + \sum^{n}_{i=2} (c_{i-1} \prod^{n}_{k=i+1} f_k) \epsilon_i.
    \end{split}
\end{equation*}
Thus, the parameter set $\{h_i\}$ are chosen as  $h_1 = \prod^n_{i=2} f_i$, $h_i = c_{i-1} \prod^n_{k=i+1} f_k$ for $i \geq 2$.
\end{proof}

Lemma \ref{lem_multi} concludes that the multiplication error is the sum of initial errors of all factors, contributing to derive the accumulative error of monomial species $\{Q^l_{\tilde{s}\pm\pm}, Q^l_{v\pm}\}$ for $\tilde{s} \in I_s = I^1_s \bigcup I^2_s = \{1,2,3,4\} \bigcup \{7,8\}$ and $v \in I_v = I^1_v \bigcup I^2_v =\{5,6\} \bigcup \{9\}$ with concentrations $q^l_{\tilde{s}\pm\pm}(t), q^l_{v\pm}(t)$ in $\mathcal{O}^1_{27}$. Taking $Q^l_{1++}$ as an example, we have the dynamical equations for computing $q^l_{1++}=e^+_l\mathcal{W}^+_{2_{11}}y_l(1-y_l)\tilde{\Upsilon}_{1l} (1-\tilde{\Upsilon}_{1l}) \Xi_{1l}$ that
\begin{gather*}
    \begin{split}
\dot{u}_1(t) & = e^+_l(t)p^+_{1l}(t) - u_1(t),~ \dot{u}_2(t) = \tilde{y}^l(t)s^l_y(t) - u_2(t),\\
\dot{u}_3(t) & = s^l_{p_i}(t)w^+_{2_{11}}(t) - u_3(t),~ \dot{\tilde{u}}_1(t) = u_1(t)u_2(t) - \tilde{u}_1(t),\\
\dot{\tilde{u}}_2(t) & = s^l_1(t)u_3(t) - \tilde{u}_2(t),~\dot{q}^l_{1++}(t)=\tilde{u}_1(t)\tilde{u}_2(t)-q^l_{1++}(t),
    \end{split}
\end{gather*}
where $u_i(t), \tilde{u}_j(t)$ denotes concentrations of intermediate species in $\mathcal{O}^1_{27}$ and we direct readers to Part \Rmnum{1} for more details. This approach is effective since initial errors of factors constituting $\bar{q}^l_{1++} = e^+_l(T) \tilde{p}_{1l}(T) \tilde{y}_l(T) s^l_y(T) s^l_{p_i}(T) w^+_{2_{11}}(0) s^l_1(T)$ have been previously derived.
Here, we introduce $par^{\pm}(t) \in \mathbb{R}^{3 \times 3}_{\geq 0}$ to denote the concentration matrix of negative species pairs $\{Par^{\pm}_{\tilde{s}}, Par^{\pm}_{v}\}$ in $\mathcal{O}^2_{27}$ with the correspondence between subscripts $\tilde{s}$ and $ij$ in $par^{\pm}(t)$ following $1,3,7 \leftrightarrow 11, 12, 13$, $2,4,8 \leftrightarrow 21, 22, 23$, $5,6,9 \leftrightarrow 31, 32, 33$, respectively. We illustrate dynamical equations of $\mathcal{O}^2_{27}$ by species$\{Par^{\pm}_1\}$
\begin{gather*}
    \begin{split}
\dot{par}^{\pm}_{11}(t) = \sum^{\tilde{p}}_{l=1} (q^l_{1+\pm}(t)+q^l_{1-\mp}(t)) - par^{\pm}_{11}(t).
    \end{split}
\end{gather*}
Then, by combining the exponential convergence characterization of type-\Rmnum{2} CRN (see Lemma 3 in \cite{fan2023automatic}) and incorporating the addition component, we present total errors of the positive $-\nabla \mathcal{E}^1_+$ and negative parts $-\nabla \mathcal{E}^1_-$ of negative gradients, computed in $\mathcal{M}^{le}_{1}$, for all weights. 
\begin{proposition}
Let $-\nabla \mathcal{E}^1_{\pm} = ((par^{1}_{\pm,1})^{\top}, (par^{1}_{\pm,2})^{\top})^{\top}$ with $par^{1}_{\pm,1} \in \mathbb{R}^{2 \times 3}_{\geq 0}, par^{1}_{\pm,2} \in \mathbb{R}^{1 \times 3}_{\geq 0}$ be the standard positive and negative part of $-\nabla \mathcal{E}^1$. Then, there exist the positive parameter set $\{\mathcal{L}^q_{\tilde{s}}\}^{6}_{q=1}$, $\{\mathcal{L}^{7}_{\tilde{s},i}\}^{2}_{i=1}$, $\{\mathcal{L}^8_{\pm, \tilde{s}}\}$ and $\tilde{\gamma} = 1-\frac{1}{e}$ for species $\{Par^{\pm}_{\tilde{s}}\}$ in $\mathcal{O}^2_{27}$ with $\tilde{s} \in I_s$  such that 
\begin{equation}
    \begin{split}
\vert & par^{\pm}_{1_{ij}}(T)  - par^1_{\pm,1_{ij}} \vert < \mathcal{L}^1_{\tilde{s}} e^{-\tilde{K}^{\ast} T} + \mathcal{L}^2_{\tilde{s}} Te^{-\alpha_3 T}  \\
&
+ \mathcal{L}^3_{\tilde{s}} T^2 e^{-T} + \mathcal{L}^4_{\tilde{s}} T^2e^{-\alpha T} + \mathcal{L}^5_{\tilde{s}} Te^{- n ^1 T}  \\
& + \mathcal{L}^6_{\tilde{s}} e^{- \tilde{n}^1  T} +\mathcal{L}^{7}_{\tilde{s},i} e^{-K^{i,\ast} T}
+ \mathcal{L}^8_{\pm, \tilde{s}} e^{-\tilde{\gamma} T}.
    \end{split}
\label{eq:Prop51}
\end{equation}
Concerning species $\{Par^{\pm}_{v}\}$ with $v \in I_v$, we can find another positive parameter set $\{\tilde{\mathcal{L}}^q_{\pm,v}\}^7_{q=1}$ for the error estimation 
\begin{equation}
    \begin{split}
& \vert  par^{\pm}_{2_{1j}}(T) - par^1_{\pm,2_{1j}} \vert < \tilde{\mathcal{L}}^1_{\pm,v} e^{-\tilde{K}^{\ast} T} + \tilde{\mathcal{L}}^2_{\pm,v} Te^{-\alpha_3 T} \\
 &  +  \tilde{\mathcal{L}}^3_{\pm,v} T^2 e^{-T} + \tilde{\mathcal{L}}^4_{\pm,v} T^2e^{-\alpha T} + \tilde{\mathcal{L}}^5_{\pm,v} Te^{-n^1 T} \\
 & + \tilde{\mathcal{L}}^6_{\pm,v} e^{-\tilde{n}^1 T}  + \tilde{\mathcal{L}}^7_{\pm,v} e^{-\tilde{\gamma} T}
    \end{split}
\label{eq:prop52}
\end{equation}
with $par^{\pm}(t)=((par^{\pm}_1(t))^{\top}, (par^{\pm}_2(t))^{\top})^{\top}$.
\end{proposition}
\begin{proof} 
Firstly, Lemma \ref{lem_multi} ensures the existence of positive coefficients $\{h^r_{\tilde{s}\pm\pm}\}^{7}_{r=1}$ for $\{Q^l_{\tilde{s}\pm\pm}\}_{\tilde{s} \in I^1_s}$, $\{h^r_{\tilde{s}\pm\pm}\}^{6}_{r=1}$ for $\{Q^l_{\tilde{s}\pm\pm}\}_{\tilde{s} \in I^2_s}$, $\{h^r_{v\pm}\}^4_{r=1}$ for $\{Q^l_{v\pm}\}_{v \in I^1_v}$, and $\{h^r_{v\pm}\}^3_{r=1}$ for $\{Q^l_{v\pm}\}_{v \in I^2_v}$ to have the accumulative error estimation formulas.
Then, note that $\vert \bar{q}^l_{\tilde{s}++} - q^l_{\tilde{s}++} \vert = \vert \bar{q}^l_{\tilde{s}+-} - q^l_{\tilde{s}+-} \vert, \vert \bar{q}^l_{\tilde{s}--} - q^l_{\tilde{s}--} \vert = \vert \bar{q}^l_{\tilde{s}-+} - q^l_{\tilde{s}-+} \vert$ where $\tilde{s} \in I_s$ since $w^{\pm}_2(0)$ is accurate for all $\bar{q}^l_{\tilde{s}\pm \pm}$ when $m=1$. Substituting total errors of $E^{\pm}_l, S^l_y, Y^l, S^l_{p_i}, P^l_i, S^l_i$ with $i=\{1,2\}$ deduced for upstream modules into the aforementioned formulas, one can separately derive $\{(\mathcal{L}^q_{\pm, \tilde{s}})^l\}^6_{q=1}$,  $\{(\mathcal{L}^{7}_{\pm,\tilde{s},i})^l\}^2_{i=1}$ to be coefficients of various exponential items in the upper bound of $\vert \bar{q}^l_{\tilde{s}+\pm} - q^l_{\tilde{s}+\pm} \vert$ and $\vert \bar{q}^l_{\tilde{s}-\pm} - q^l_{\tilde{s}-\pm} \vert$. 
Furthermore, it holds for $\mathcal{W}^1_1$ that
\begin{equation*}
    \begin{split}
\vert \bar{par}^{\pm}_{1_{ij}} - par^1_{\pm,1_{ij}} \vert & < \sum^{\tilde{p}}_{l=1} \vert \bar{q}^l_{i+\pm} - q^l_{i+\pm} \vert + \vert \bar{q}^l_{i-\mp} - q^l_{i-\mp} \vert.
    \end{split}
\end{equation*}
Let $(\mathcal{L}^{q}_{\tilde{s},i})^l = (\mathcal{L}^{q}_{+,\tilde{s},i})^l + (\mathcal{L}^{q}_{-,\tilde{s},i})^l$ for $q$ and $i$ if any.  After summing over all $l$, we select the final parameters for each type of item by retaining the exponential power maximum point and $\tilde{p}$-fold of the coefficient maximum point concerning $l$, for example, $\{\frac{\mathcal{L}^1_{\tilde{s}}}{\tilde{p}},\tilde{K}^{\ast}\} = \mathop{\arg\max}\limits_{l} \{(\mathcal{L}^1_{\tilde{s}})^l e^{-(\tilde{K})^lT}\}$. Finally, $\mathcal{L}^8_{\pm, \tilde{s}}$ and $\tilde{\gamma}$ is determined by the current error of $\{Par^{\pm}_{\tilde{s}}\}$ assured by Lemma 3 in \cite{fan2023automatic}.

A similar analysis is applied for $\mathcal{W}^1_2$ that there exists $\{(\tilde{\mathcal{L}}^q_{\pm,v})^l\}^6_{q=1}$ working for $\vert \bar{q}^l_{\pm,v} - q^l_{\pm,v} \vert$, $\{\tilde{\mathcal{L}}^j_{\pm,v}\}^6_{j=1}$ contributing to accumulative errors of $\{Par^{\pm}_v\}$ based on
\begin{equation*}
    \begin{split}
\vert \bar{par}^{\pm}_{2_{1j}} - par^{1}_{\pm, 2_{1j}} \vert  < \sum^{\tilde{p}}_{l=1} \vert \bar{q}^l_{\pm,v} - q^l_{\pm,v} \vert
    \end{split}
\end{equation*}
and $\{\tilde{\mathcal{L}}^7_{\pm,v}\}, \tilde{\gamma}$ for their current errors. 
\end{proof}
 
Estimating the distance between the real iterative sequence $w^{\pm}(T) = ((w^{\pm}_1(T))^{\top},(w^{\pm}_2(T))^{\top})^{\top} \in \mathbb{R}^{3 \times 3}_{\geq 0}$ and standard sequence $\mathcal{W}^2_{\pm}$ highlights the difference in the training performance of two neural networks. As displayed in Fig.\ref{fig:8}, the deviation in computing negative gradients by $\mathcal{M}^{le}_1$ and storing current weight/bias concentrations in species $\{\Gamma^{\pm}_{i}\}^9_{i=1}$ by $\mathcal{O}^2_7$ \cite{fan2023automatic} leads to the accumulative error on incremental weight/bias species $\{\Delta W^{\pm}_i\}^6_{i=1}$, $\{\Delta B^{\pm}_j\}^3_{j=1}$ in $\mathcal{M}^{le}_2$ and weight/bias species. Here, $\Delta w^{\pm}(t) \in \mathbb{R}^{3 \times 3}_{\geq 0}$ is regarded as the concentration matrix of incremental species 
with $\Delta w^{\pm}_{1\cdot}(t)$ to identify $\Delta W^{\pm}_1, \Delta W^{\pm}_3, \Delta B^{\pm}_1$'s, $\Delta w^{\pm}_{2\cdot}(t)$ to denote $\Delta W^{\pm}_2, \Delta W^{\pm}_4, \Delta B^{\pm}_2$'s and $\Delta w^{\pm}_{3\cdot}(t)$ to indicate $\Delta W^{\pm}_5, \Delta W^{\pm}_6, \Delta B^{\pm}_3$'s. The initial deviation of $\Gamma^{\pm}_{i}$ is  
\begin{equation*}
    \begin{split}
\Vert \gamma^{\pm} (T) - \mathcal{W}^1_{\pm} \Vert & < \Vert \gamma^{\pm}(T) - w^{\pm}(0) \Vert + \Vert w^{\pm}(0) - \mathcal{W}^1_{\pm} \Vert \\
& < \Vert \gamma^{\pm}(T) - w^{\pm}(0) \Vert < \mathcal{I}^{\pm} e^{-T},
    \end{split}
\label{eq:gamma_(T)}
\end{equation*}
where $\gamma^{\pm}(t) \in \mathbb{R}^{3 \times 3}_{\geq 0}$ and
$\mathcal{I}^{\pm} = \Vert w^{\pm}(0) - \gamma^{\pm}(0)\Vert$. By combining the established exponential convergence of $\mathcal{M}^{le}_2$ \cite{fan2023automatic}, and 
denoting the upper bound of \eqref{eq:Prop51} and \eqref{eq:prop52} by $\{\mathcal{L}^{qe}_{\pm,\tilde{s},i}\}^8_{q=1}, \{\tilde{\mathcal{L}}^{qe}_{\pm,v}\}^7_{q=1}$,
we present the following theorem to conclude the overall deviation in updating weights during the first iteration. 
\begin{theorem}
Assigning appropriate initial concentrations to all species in the BFCNN, for $m=1$, there exist the parameter sets $\{\tilde{\mathcal{R}}^{\tilde{q}}, \tilde{\mathcal{R}}^{q}_{\pm}\}^8_{\tilde{q} \neq \{3,8\}}$, $\{\tilde{\mathcal{S}}^{q}_{\pm}\}^7_{q=1} > 0$ such that
\begin{gather}
    \begin{split} \label{eq:estiW_1}
\Vert & w^{\pm}_{1}(T) - \mathcal{W}^2_{\pm,1} \Vert ^{m=1} < \Vert w^{\pm}_{1}(0) -\mathcal{W}^1_{\pm,1} \Vert^{m=1} + \tilde{\mathcal{R}}^1 e^{-\tilde{K}^{\ast}T} \\
& + \tilde{\mathcal{R}}^2 Te^{-\alpha_3 T} + \tilde{\mathcal{R}}^3_{\pm} T^2 e^{-T} + \tilde{\mathcal{R}}^4 T^2e^{-\alpha T} + \tilde{\mathcal{R}}^5 Te^{-n^1 T}  \\
& + \tilde{\mathcal{R}}^6 e^{-\tilde{n}^1 T} +\tilde{\mathcal{R}}^7 e^{-K^{\ast}T} + \tilde{\mathcal{R}}^8_{\pm} e^{-\tilde{\gamma} T},
\end{split}\\
\begin{split}\label{eq:estiW_2}
\Vert & w^{\pm}_{2}(T) - \mathcal{W}^2_{\pm,2} \Vert ^{m=1} < \Vert w^{\pm}_{2}(0) -\mathcal{W}^{1}_{\pm,2} \Vert^{m=1} + \tilde{\mathcal{S}}^1_{\pm} e^{-\tilde{K}^{\ast}T} \\ 
& + \tilde{\mathcal{S}}^2_{\pm} Te^{-\alpha_3 T} + \tilde{\mathcal{S}}^3_{\pm} T^2 e^{-T} + \tilde{\mathcal{S}}^4_{\pm} T^2e^{-\alpha T} + \tilde{\mathcal{S}}^5_{\pm} Te^{-n^1 T}  \\
& + \tilde{\mathcal{S}}^6_{\pm} e^{-\tilde{n}^1 T}+ +\tilde{\mathcal{S}}^7_{\pm}  e^{-\tilde{\gamma} T}.
    \end{split}
\end{gather}
\end{theorem}
\begin{proof}
The equations of the module $\mathcal{M}^{le}_2$ are
\begin{equation}
\begin{split}
\dot{\Delta w^{\pm}}(t)& =\eta par^{\pm} (T) - \Delta w^{\pm}(t), \\
\dot{w}^{\pm} (t) & = \gamma^{\pm}(T) + \Delta w^{\pm}(t)-w^{\pm}(t),
\end{split}
\label{eq:w,deltaw}
\end{equation}
where and we have the solution below
\begin{equation*}
    \begin{split}
\Delta w^{\pm}(t) & =\Delta w^{\pm}(0)e^{-t} +\eta par^{\pm}(T) (1-e^{-t}),\\
w^{\pm}(t) & = \gamma^{\pm}(T) + \eta par^{\pm}(T) + (w^{\pm}(0)- \gamma^{\pm}(T) \\
& - \eta par^{\pm}(T))e^{-t} +(\Delta w^{\pm}(0)-\eta par^{\pm}(T) )te^{-t}.
    \end{split}
\end{equation*}
Then, let $\{\mathcal{L}^{qe,1}_{\pm}\}= \mathop{\max}\limits_{\tilde{s} \in I_s, i} \{\mathcal{L}^{qe}_{\pm, \tilde{s},i}\}$, $\{\mathcal{L}^{qe,2}_{\pm}\}= \mathop{\max}\limits_{v \in I_v} \{\tilde{\mathcal{L}}^{qe}_{\pm, v}\}$ for $i = \{1,2\}$, and denote the first two rows and the last row of $\gamma^{\pm}(t), \Delta w^{\pm}(t)$ by $\gamma^{\pm}_1(t), \Delta w^{\pm}_1(t)$ and $\gamma^{\pm}_2(t), \Delta w^{\pm}_2(t)$, respectively.
We derive the current error 
\begin{equation*}
 \begin{split}
\Vert  w^{\pm}_{j}(T) - \bar{w}^{\pm}_{j} \Vert & = \Vert w^{\pm}_{j}(T)-(\gamma^{\pm}_j(T) + \eta par^{\pm}_j(T))\Vert~~\\
& \leq \mathcal{T}^{\pm}_{j} e^{-T} + \Delta  \mathcal{T}^{\pm}_{j} Te^{-T}, \\
\Vert  \Delta w^{\pm}_{j}(T) - \bar{\Delta w}^{\pm}_{j} \Vert & = \Vert \Delta w^{\pm}_{j}(T) -\eta par^{\pm}_{j}(T)  \Vert \\
& < \Delta \mathcal{T}^{\pm}_{j} e^{-T},
\end{split}   
\end{equation*}
where $j \in \{1,2\}$, and $\mathcal{T}^{\pm}_{j} = \Vert w^{\pm}_{j}(0) - \gamma^{\pm}_{j} - \eta par^{\pm}_j \Vert + \Vert w^{\pm}_{j}(0) - \mathcal{W}^1_{\pm,j} \Vert^{m=1} + \mathcal{I}^{\pm} e^{-T} +\eta \{L^{qe,j}_{\pm}\}$, $\Delta \mathcal{T}^{\pm}_{j} = \Vert \Delta w^{\pm}_{j}(0)- \eta par^{\pm}_j \Vert + \eta \{L^{qe,j}_{\pm}\}$. Thus, the total error is constituted by
\begin{equation*}
    \begin{split}
\Vert  w^{\pm}_{j}(T) - \mathcal{W}^2_{\pm,j} \Vert & \leq \Vert w^{\pm}_j(0) - \mathcal{W}^1_{\pm,j} \Vert^{m=1} + \mathcal{T}^{\pm}_j e^{-T} ~~\\
&~~~ + \Delta \mathcal{T}^{\pm}_j Te^{-T} + \eta \{L^{qe}_{\pm,j}\} + \mathcal{I}^{\pm} e^{-T}, ~~\\ 
\Vert \Delta w^{\pm}_j(T) - \Delta w^{\pm}_j  \Vert & \leq \Delta \mathcal{T}^{\pm}_j e^{-T} + \eta \{L^{qe}_{\pm,j}\}.
    \end{split}
\end{equation*}
Notably, $\Vert w^{\pm}_j(0) - \mathcal{W}^1_{\pm,j} \Vert^{m=1} =0 $ is caused by the vertical propagation from $\gamma^{\pm}(T)$. Assuming that $T>1$ without loss of generality, coefficients after concatenation in $\Vert w^{\pm}_{j}(T) - \mathcal{W}^2_{\pm,1} \Vert^{m=1}$ become $\tilde{\mathcal{R}}^3_{\pm} = \eta \mathcal{L}^{3,1} + \mathcal{I}^{\pm} + \mathcal{T}^{\pm}_1 + \Delta \mathcal{T}^{\pm}_1$, $\tilde{\mathcal{S}}^3_{\pm} = \eta \mathcal{L}^{3,2}_{\pm} + \mathcal{I}^{\pm} + \mathcal{T}^{\pm}_{2} + \Delta \mathcal{T}^{\pm}_{2}$, and $\tilde{\mathcal{R}}^{\tilde{q}} = \eta \mathcal{L}^{\tilde{q},1}, \tilde{\mathcal{R}}^{8}_{\pm} = \eta \mathcal{L}^{8,1}_{\pm}$, $ \tilde{\mathcal{S}}^q_{\pm}=\eta \mathcal{L}^{q,2}_{\pm}$, with $\{ \mathcal{L}^{\tilde{q},1}\}^7_{\tilde{q}=1}, \{\mathcal{L}^{8,1}_{\pm}\}$ and $\{\mathcal{L}^{q,2}_{\pm}\}^7_{q=1}$ representing coefficients of $\{\mathcal{L}^{qe,1}_{\pm}\}$ and $
\{\mathcal{L}^{qe,2}_{\pm}\}$, separately. Additionally, exponential powers remain unchanged except for $K^{\ast} = \mathop{\min}\{K^{1,\ast}, K^{2,\ast}\} $. Hence, \eqref{eq:estiW_1}, \eqref{eq:estiW_2} are obtained.
\end{proof}
\subsection{Error Convergence Order}
The general depiction of the algorithm convergence in silicon refers to deducing the concrete relation inequality between distance ($\Vert \mathcal{W}^m - \mathcal{W}^{\ast} \Vert$ or $\vert \mathcal{E}^m - \mathcal{E}^{\ast} \vert$) and iteration number $m$. From the relation, the convergence order can be determined, which reflects the runtime that an algorithm needs to achieve a specified level of accuracy, i.e., the capability, dependent on $m$, to reduce the distance.  

Until now, we have derived an error estimation for the initial iteration, reasonably described by parameter $T$, for updating weights in BFCNN compared to that in FCNN. The relationship \eqref{eq:estiW_1} and \eqref{eq:estiW_2} ensures that the intrinsic convergence of the algorithm is preserved in the first iteration through our BCRN implementation, given that the realization error converges regarding $T$. Additionally, it should be noted that the phase length of the chemical oscillator affects not only the runtime of BFCNN but also training accuracy. Then, we give the following technical result to lay the foundation for characterizing the convergence order.  
\begin{lemma} \label{eq:lem_for_ expone}
For arbitrary given positive constant $a, b$, there must exist $\mathcal{U}, v>0$ such that $t^a e^{-bt} < \mathcal{U}e^{-vt}$ for $t \geq 0$.
\end{lemma}
\begin{proof}
To find the appropriate constant $U,v$, we consider the function and its stationary point
\begin{equation*}
g(t)= (b-v)t + \ln{\mathcal{U}} -a \ln{t}, ~ t^{\star} = \frac{a}{b-v}.
\end{equation*}
Then, we choose $v<b$ otherwise the inequality will not hold when $t$ becomes large enough. Note that $t^{\star}$ is the minima so $\mathcal{U},v$ needs to meet $g(t^{\star})>0$, that is,
\begin{equation*}
    b-v>e^{\ln{a} -1-\frac{\ln{\mathcal{U}}}{a}}.
\end{equation*}
To make $v>0$, $\mathcal{U},v$ should be chosen for $\ln{a}-1-\frac{\ln{\mathcal{U}}}{a} < \ln{b}$ and $v<b-e^{\ln{a} -1-\frac{\ln{\mathcal{U}}}{a}}$. Therefore, the valid range of $\mathcal{U},v$ is 
\begin{equation*}
\mathcal{U} \geq e^{a(\ln{\frac{a}{b}}-1)+\delta},~0<v \leq b(1-e^{-\frac{\delta}{a}})
\end{equation*}
for any positive relaxation constant $\delta$.
\end{proof}

\begin{remark}
Lemma \ref{eq:lem_for_ expone} guarantees that $\{\mathcal{U}^r_{q,\pm}, v^r_q\}^5_{q=2}$, $\{\mathcal{U}^s_{q,\pm}, v^s_q\}^5_{q=2}$ can be determined for the second to fifth items in 
\eqref{eq:estiW_1}, \eqref{eq:estiW_2}, respectively, such that 
\begin{equation*}
    \begin{split}
\Vert & w^{\pm}_{1}(T) - \mathcal{W}^2_{\pm,1} \Vert ^{m=1} < \Vert w^{\pm}_{1}(0) -\mathcal{W}^1_{\pm, 1} \Vert ^{m=1} + \tilde{\mathcal{R}}^{\pm}e^{-v^r T}, \\
\Vert & w^{\pm}_{2}(T) - \mathcal{W}^2_{\pm,2} \Vert ^{m=1} < \Vert w^{\pm}_{2}(0) -\mathcal{W}^1_{\pm,2} \Vert^{m=1} + \tilde{\mathcal{S}}^{\pm}e^{-v^s T},
    \end{split}
\end{equation*}
where $\tilde{\mathcal{R}}^{\pm}$ and $\tilde{\mathcal{S}}^{\pm}$ are the sum of revised coefficients in \eqref{eq:estiW_1}, \eqref{eq:estiW_2}, respectively, and $v^r=\mathop{\min}\{\tilde{K}^{\ast}, \tilde{n}^1, K^{\ast}, \tilde{\gamma},\{v^r_q\}^5_{q=2}\}$, $v^s=\mathop{\min} \{\tilde{K}^{\ast}, \tilde{n}^1, \tilde{\gamma},\{v^s_q\}^5_{q=2}\}$. Due to $\Vert w^{\pm}_{j}(0) -\mathcal{W}^1_{\pm,j} \Vert ^{m=1}=0$ for $j=1,2$, the initial realization error 
for training weights is controlled by a negative exponential regarding $T$ with rate $v^r, v^s$. In this case, we call the realization error of neural networks through BFCNN possessing \textit{Exponential Order Convergence}, which ensures that a small increment on $T$ will significantly reduce the approximation convergence error. Moreover, larger exponents $v^r, v^s$ lead to greater reductions within a fixed time interval. Consequently, the training in BFCNN can converge to the optimal weights within a given accuracy at a relatively fast rate compared to the non-exponential error form.
\end{remark}
\section{Cumulative Error of the $m_{th}$ $(m \geq 2)$ Iteration}
\label{sec:error_m iteration}
The error estimation for the first iteration clearly illustrates the current error arising from each module and the vertical propagation pathways from the upstream to the downstream modules. The horizontal propagation from previous iterations affects the realization error when BFCNN enters the $m$th ($m \geq 2$) iteration, which may increase the original coefficients and introduce lower-order error entries, thereby causing a greater accumulation of total error over the number of iterations and modules. In this section, we aim to show the general error upper bound formula of any $m$th iteration through recursion. 

Given that the continual operation of BFCNN before achieving the preset accuracy makes experimenters unable to calibrate initial concentrations before every training round, the lateral propagation causes endpoints error derived in the $m-1$th round to be initial errors of the corresponding species for $m \geq 2$. Furthermore, the clear-out module \cite{fan2023automatic} forces the concentrations of species $\{I^l_y, I^l_{pi}\}$ to end up being near one in the first round, namely
\begin{equation*}
\begin{split}
\vert i^l_{y}(2T)-1 \vert & < \vert 1-i^l_{y}(T) \vert e^{-T}, \\
\vert i^l_{p_i}(2T)-1 \vert & < \vert 1-i^l_{p_i}(T) \vert e^{-T},
\end{split}
\end{equation*}
and for species $X \in \mathcal{S}^f, \tilde{X} \in \mathcal{S}^p$ cleared out, concentrations to follow
\begin{equation*}
\vert x(3T) \vert < \vert x(2T) \vert e^{-T}, ~\vert \tilde{x}(2T) \vert < \vert \tilde{x}(T) \vert e^{-T},
\end{equation*}
where $\mathcal{S}^f = \{P^{\pm}_{il}, Y^{\pm}_l\}$ and $\mathcal{S}^p = \{Y^l_s, Y^l_e, \tilde{Y}^l, P^l_{is}, \tilde{P}^l\}$. Based on the specific initial deviation emerging from $m=1$, the horizontal propagation pathways can be investigated by spanning from $m=1$ to $m=2$. Hereafter, coefficients that appeared previously are indicated by appending superscripts $m=1$, and the variable $x^m(t) = x(t+o(m-1)T)$ with $o$ denoting the number of phases in which species $X$ participate in one period. Now, we showcase that the assignment module preserves convergence orders in subsequent iterations.
\begin{proposition}\label{prop:ass_2}
If $c^m = (e_{1+m\tilde{p}-jp}, \cdots, e_{(m+1)\tilde{p}-jp})^{\top}$ $\tilde{c}^m = \mathds{O}_{p \times \tilde{p}}$, and $s^m=(\chi_{\cdot (m-1)\tilde{p}+1-jp}, \cdots, \chi_{\cdot m\tilde{p}-jp})^{\top}$ with $j= \lfloor \frac{\tilde{p}m}{p} \rfloor$ ($m\nmid \frac{p}{\tilde{p}}$) and  $j= \lfloor \frac{\tilde{p}m}{p} \rfloor-1 $ ($m\mid \frac{p}{\tilde{p}}$), there exist parameters sequences $m^m_1$, $ m^m_3, \tilde{m}^m_3$ for $m \geq 2$ such that
\begin{equation} 
    \begin{split}
\Vert s^m(T)-s^m \Vert_m & \leq  m^m_1 e^{-T},\\
\Vert c^m(2T) -c^m \Vert_m & \leq  m^m_3 e^{-kT},\\
 \Vert \tilde{c}^m(2T) -\mathds{O}_{p \times \tilde{p}} \Vert_m & \leq \tilde{m}^m_3 e^{-kT}.
    \end{split}
\label{eq:prop7}
\end{equation}
\end{proposition}
\begin{proof}
Firstly, we derive the case of $m=2$ by utilizing Proposition \ref{prop:assign_1}. In $\mathcal{M}^a_1$, initial deviation of $s^2(0), c^2(0)$ results in accumulative errors $\Vert \bar{s}^2 - s^2 \Vert _{m=2} \neq 0$, i.e., 
\begin{equation*}
    \begin{split}
\Vert \chi (0) c^2(0) - \chi c^1 \Vert_{m=2} & \leq  \Vert \chi \Vert \Vert c^1(2T) -c^1  \Vert \leq \Vert \chi \Vert m^1_3 e^{-kT},
    \end{split}
\end{equation*}
and altering the upper bound of the current error $ \Vert s^2(T) - \bar{s}^2 \Vert = \Vert s^2(0)-\chi (0)c^2(0) \Vert_{m=2} e^{-T}$ through
\begin{equation*}
    \begin{split}
\Vert s^2(0)-\chi(0) c^2(0) \Vert_{m=2} & \leq \Vert s^1(2T) - \{s\}_{\mathcal{O}^1_{23}} \Vert + \Vert \chi \Vert m^1_3 e^{-kT} \\
&~~~ + \Vert \{s\}^1_{\mathcal{O}_{23}} - \chi c^1 \Vert.    
    \end{split}
\end{equation*}
By substituting the corresponding items, one can have
\begin{equation*}
    \begin{split}
\Vert & s^2(T) - s^2 \Vert \leq \Vert \{s\}^1_{\mathcal{O}_{23}} -\chi c^1 \Vert e^{-T} + m^1_1 e^{-2T} \\
& + M^1_x e^{-(\tilde{K}^1+1)T} + \{E\}^1 e^{-T} +  \Vert \chi \Vert m^1_3 e^{-kT}(1+e^{-T}). 
    \end{split}
\end{equation*}
Here, $M^1_x, \tilde{K}^1$ denotes the maximum values concerning $l$ in the first round. Note that we can find a uniform upper bound $\Theta_1$ of $m^1_1 e^{-T} + M^1_x e^{-\tilde{K}^1T} + \{E\}^1$ for $\forall T$ and set the rate constant $k \geq 1$,then there exists a uniform upper bound
\begin{equation*}
    \begin{split}
\Vert s^2(T) - s^2 \Vert_{m=2} \leq m^2_1 e^{-T} 
    \end{split}
\end{equation*}
with $m^2_1 = \Theta_1 + 2\Vert \chi \Vert m^1_3 + \Vert \{s\}^1_{\mathcal{O}_{23}} -\chi c^1 \Vert $. Similarly in $\mathcal{M}^a_2$, the initial deviations of $\Vert c^2(0) \Vert$ and $\Vert \tilde{c}^2(0) \Vert$ are presented as 
\begin{equation*}
    \begin{split}
\Vert c^2(0) \Vert & =\Vert c^1(2T) \Vert \leq \Vert c^1(2T) -c^1 \Vert + \Vert c^1 \Vert  \\
& \leq m^1_3 e^{-kT} + \Vert c^1 \Vert, \\
\Vert \tilde{c}^2(0) \Vert & = \Vert \tilde{c}^1(2T) \Vert \leq \tilde{m}^1_3 e^{-kT}.
    \end{split}
\end{equation*}
Then, we derive the new upper bound in $\mathcal{M}^a_2$
\begin{equation*}
    \begin{split}
\Vert c^2(T) - \mathds{O}_{p \times \tilde{p}} \Vert_{m=2} \leq m^2_2 e^{-kT}, ~ \Vert \tilde{c}^2(T) - c^1 \Vert_2 \leq \tilde{m}^2_2 e^{-kT},     
    \end{split}
\end{equation*}
where $m^2_2 = m^1_3 + \Vert c^1 \Vert$, $\tilde{m}^2_2 = m^2_2 + \tilde{m}^1_3 + m^1_3$. 

For $\mathcal{M}^a_3$, Fig.\ref{fig:5} shows that accumulative errors are solely by vertical propagation from $c^2(T), \tilde{c}^2(T)$. Finally, following Proposition \ref{prop:assign_1}, total errors of the assignment module are
\begin{equation*}
    \begin{split}
\Vert c^2(2T) - c^2 \Vert_{m=2} \leq m^2_3 e^{-kT}, ~ \Vert \tilde{c}^2(2T) - \tilde{c}^2 \Vert_{m=2} \leq \tilde{m}^2_3 e^{-kT}.    
    \end{split}
\end{equation*}
Here, $(m^2_3)^l=\Vert \mathcal{K}^l \Vert(\Vert c^1_{\cdot l} \Vert +2 \tilde{m}^2_2)+m^2_2$ with $m^2_3  = \mathop{\max}_{l} \{(m^2_3)^l\} $, and $\tilde{m}^2_3 = \tilde{m}^2_2 + \Vert c^1 \Vert$. The horizontal propagation from $m-1$th to $m$th ($m \geq 2$) on the assignment module follows the same pattern as that from $m=1$ to $m=2$, assuring that it only causes errors of the higher or equal order. Consequently, the error formula \eqref{eq:prop7} can be deduced by treating $m-1$th as the initial deviation through mathematical induction.  
\end{proof}

This result implies that the error convergence order of the assignment module is always maintained regardless of iteration numbers as long as the rate constant is designed as $k>1$. Despite that, a significant impact on the total error estimation of the feedforward module arises from the initial deviation of weights. Next, we show the difference in the general feedforward output errors compared to Proposition \ref{prop-psp} and Corollary \ref{coro_1} by deriving the specific estimation formula.   
\begin{proposition} \label{prop:feed_m}
Using $\{\tilde{\mathcal{R}}^{qe}_{\pm}\}^{1}, \{\tilde{\mathcal{S}}^{qe}_{\pm}\}^{1}$ to denote the exponential upper bounds of \eqref{eq:estiW_1} and \eqref{eq:estiW_2}, abbreviating $\Vert w^{\pm}_{j}(0)-\mathcal{W}^1_{\pm,j} \Vert ^{m=1}$ as $\Vert w^{\pm}_{j} \Vert^1$ with $j \in \{1,2\}$, then there exist coefficient sets $\{\tilde{\mathcal{C}}^2_r\}^2_{r=1}$, $\{\mathcal{D}^2_r\}^5_{r=1}$ such that the output error of species $\{P^l_i, Y^l\}$ during $m=2$ is estimated as
\begin{equation}
    \begin{split}
\Vert p^2(T) - \Upsilon^2 \Vert_{m=2} & \leq \{A\}^{2} + \sum^2_{r=1} \tilde{\mathcal{C}}^2_r \left[\{\tilde{\mathcal{R}}^{qe}\}^{1} + \Vert w_1 \Vert^1 \right] T^{r-1}, \\
\Vert y^2(T) - y^2 \Vert_{m=2} & \leq \{E\}^{2} + \sum^3_{r=1} \mathcal{D}^2_r \left[\{\tilde{\mathcal{R}}^{qe}\}^{1} + \Vert w_1 \Vert^1 \right] T^{r-1} \\
& ~~~ + \sum^5_{r=4} \mathcal{D}^2_r \left[\{\tilde{S}^{qe}\}^{1} + \Vert w_2 \Vert^1 \right] T^{r-4},
    \end{split}
\label{eq:prop8}
\end{equation}
where $p(t)=((p_1(t))^{\top},(p_2(t))^{\top})^{\top} \in \mathbb{R}^{2 \times \tilde{p}}$, $\{\tilde{\mathcal{R}}^{qe}\}^1 = \{\tilde{\mathcal{R}}^{qe}_+\}^{1}+\{\tilde{\mathcal{R}}^{qe}_-\}^{1}, \{\tilde{\mathcal{S}}^{qe}\}^1 = \{\tilde{\mathcal{S}}^{qe}_{+}\}^{1}+\{\tilde{\mathcal{S}}^{qe}_{-}\}^{1}$, $\Vert w_j \Vert^1 = \Vert w^{+}_{j} \Vert^1 + \Vert w^{-}_{j} \Vert^1$. 
\end{proposition}
\begin{proof}
Initial deviations originated from weight species $w^{\pm,1}(T)$ and net input species $n^{\pm,1}(3T)$ force to re-estimate $\Vert n^{\pm,2}(0) - w^{\pm,2}_1(0)s^2_a(T) \Vert$ and $\Vert w^{\pm,2}_{1}(0)s^2_a(T) - \mathcal{W}^1_{\pm,1}\Xi^2 \Vert$ in $\mathcal{M}^f_{l_1}$. Then, following a similar scaling method, there exists $\Theta^{\pm}_2$, s.t.,
\begin{equation*}
    \begin{split}
\Vert n^{\pm,2}(T) - \mathcal{N}^2_{\pm} \Vert_{m=2} \leq m^{\pm,2}_{4} e^{-T} + ms_{a}^2 \left[\{\tilde{\mathcal{R}}^{qe}_{\pm}\}^1 + \Vert W^{\pm}_{1} \Vert^1 \right],  
    \end{split}
\end{equation*}
where $m^{\pm,2}_{4} = \Theta^{\pm}_2 + \Vert \mathcal{N}^1_{\pm}-\mathcal{W}^1_{\pm,1} \Xi^2 \Vert + 2m^2_1\Vert \mathcal{W}^1_{\pm,1}\Vert$, $ms^2_{a} = 2(m^2_1 + \Vert \Xi^2 \Vert)$. Notably, $\Theta^{\pm}_2$ arising from $\vert n^{\pm,2}(0)-\mathcal{N}^1_{\pm} \vert$ can be concatenated into $e^{-T}$, and initial errors from weight species lead to appending a lower-order entry to the upper bound, referred to as \textit{tail item} (TI) consistently in the following. Vertical propagation from $\mathcal{M}^f_{l_1}$ causes TI of $\vert n^{\pm,2}_{il}(2T) - \vert \mathcal{N}^2_{il} \vert \vert$ in $\mathcal{M}^f_{n^1_1}$ to be $ms^2_a\left[\{\tilde{\mathcal{R}}^{qe}\}^1 + \Vert w_1 \Vert^1 \right]$. For $\{\mathcal{M}^f_{n^2_1}\}$, TI caused by vertical propagation from $n^{\pm,2}(2T)$ and horizontal accumulation on $p^{\pm,2}_{il}(0)$, separately, are
\begin{equation*}
    \begin{split}
\vert p^{s_1,2}_{il} (T) -\frac{1}{2} \vert_{tail} & = m^{s_1,2}_{5i} ms^2_a \left[\{\tilde{\mathcal{R}}^{qe}\}^1 + \Vert w_1 \Vert^1 \right] T, \\
\vert p^{s_2,2}_{il} (T) - 0 \vert_{tail} & = ( \{p^{s_2,1}_{il}\} + \tilde{\delta}^{s_2}_2 )e^{-T}, 
    \end{split}
\end{equation*}
where if $\mathcal{N}^2_{il}>0$, $s_1 = +, s_2=-$, otherwise $s_1 = -, s_2 =+$, and $p^{\pm,2}_{il}(0)$ also perturbs upwards the coefficients $m^{\pm,2}_{5i} = \frac{1}{2} + \tilde{\delta}_1$.
As for $\mathcal{M}^f_{n^3_1}$, the additional TI is derived as 
\begin{equation*}
    \begin{split}
\vert p^{s_1,2}_{il} (2T) - \Upsilon^2_{il} \vert_{tail} & = \frac{1}{4} (ms^2_a)\left[\{\tilde{\mathcal{R}}^{qe}\}^1 + \Vert w_{1} \Vert^1 \right] \\
& + m^{s_1,2}_{6i}m^{s_1,2}_{5i}ms^2_a \left[\{\tilde{\mathcal{R}}^{qe}\}^1 + \Vert w_{1} \Vert^1 \right] T, \\ 
\vert p^{s_2,2}_{il} (2T) - 0 \vert_{tail} & = (\{p^{s_2}_{il}\}^1_{\mathcal{O}_{13}} + \tilde{\delta}^{s_2}_2 )e^{-T}.
    \end{split}
\end{equation*}
Estimating $\vert p^2_{il}(T) - \Upsilon^2_{il} \vert_2$ following Proposition \ref{prop-psp}, the final TI in the first layer is derived from the vertical propagation as 
\begin{equation*}
\begin{split}
\vert p^2_{il}(T) - \Upsilon^2_{il}\vert_{tail} = 2 \vert p^{s_1,2}_{il} (2T) - \Upsilon^2_{il} \vert_{tail}
\end{split}
\end{equation*}
since horizontal accumulation on $p^2_{il}(0)$ alters the coefficients of $Te^{-T}$ and the order of entry $\{p^{s_2}_{il}\}_{\mathcal{O}^1_{13}} + \tilde{\delta}^{s_2}_2$ is not increased by operations. Then, $\tilde{\mathcal{C}}^2_1=\frac{1}{2}(ms^2_a)$, $\tilde{\mathcal{C}}^2_2=2m^2_{6}m^2_{5}ms^2_a$ with $m^2_{6}m^2_{5} = \mathop{\max}_{il}\{m^{\pm,2}_{5il}m^{\pm,2}_{6il}\}$.
Considering the module $\mathcal{M}^f_{l_2}$, $\exists ~ \Theta_3>0$ such that $\Vert \tilde{n}^{\pm,2}(T) - \tilde{\mathcal{N}}^{2}_{\pm} \Vert_{tail}$ becomes 
\begin{equation*}
 mp^2_a\left[ \{\tilde{\mathcal{S}}^{qe}_{\pm}\}^1 + \Vert w^{\pm}_{2} \Vert^1 \right] + m^{\pm,2}_7\Vert p^2(T) - \Upsilon^2 \Vert_{tail},    
\end{equation*}
where $mp^2_a = (\Theta_3 + \Vert \tilde{\Upsilon}^2 \Vert)$ and $m^{\pm,2}_7 = 2\Vert \mathcal{W}^2_{\pm ,2} \Vert$. In particular, in the first layer of BFCNN, the unique factor causing TI is the vertical propagation of $\Vert n^{\pm,2}(T) -\mathcal{N}^2_{\pm} \Vert_{m=2}$. Thus, a similar argument is applied in modules $\mathcal{M}^f_{n^j_2}$ to obtain   
\begin{equation*}
\begin{split}
\vert y^{s_1,2}_l(T) - y^{2}_{l} \vert_{tail} & = \frac{1}{4} \vert \tilde{n}^{\pm,2}_{l}(2T) - \vert \tilde{\mathcal{N}}^{2}_{l} \vert \vert_{tail} \\
& ~~~ + m^{s_1,2}_{9i}m^{s_1,2}_{8i} \vert \tilde{n}^{\pm,2}_{l}(2T) - \vert \tilde{\mathcal{N}}^{2}_{l} \vert \vert_{tail}T,\\
\vert y^{s_2,2}_l(T) - 0 \vert_{tail} & = (\{y^{s_2}_l\}^l_{\mathcal{O}_{21}} + \tilde{\delta}^{s_2}_4)e^{-T},
\end{split}   
\end{equation*}
where $m^{\pm,2}_{8i} = \frac{1}{2} + \tilde{\delta}_3$, and if $\tilde{\mathcal{N}}^2_{l}>0$, $s_1 = +, s_2=-$, otherwise $s_1 = -, s_2 =+$. Consequently, coefficients $\{\mathcal{D}^2_r\}^5_{r=1}$ is available for $\Vert y^2(T) - y^2 \Vert_{m=2}$ in \eqref{eq:prop8}. 
\end{proof}

Therefore, in accordance with Proposition \ref{prop:ass_2} and \ref{prop:feed_m}, the horizontal accumulation of errors on initial concentrations of species, excluding weight species, is incapable of introducing lower-order entries in the assignment and feedforward modules. Moving forward, as shown in Fig.\ref{fig:7}, vertical propagation from $y^2_l(T), p^2_{il}(T), s^{l,2}_3(T)$ and lateral accumulation on initial concentrations of other species have compound effects on pBCRN. Due to the clear-out module designed to prevent the occurrence of lower-order than $e^{-T}$ from previous iterations by decaying species or assigning fixed concentrations to species in advance, it can be easily confirmed that tail items in this part solely arise from vertical propagation. Thus, denoting $\Vert y^2(T)-y^2 \Vert_{tail}=\{\mathcal{D}_y\}^2$, $\Vert p^2(T) - \Upsilon^2 \Vert_{tail}=\{\mathcal{D}_p\}^2$, we derive TIs regarding $\{\tilde{y}^2_{l}(T), e^{\pm,2}_l(T), s^{l,2}_y(T)\}$ to be $\{\mathcal{D}_y\}^2$ and 
regarding $\{\tilde{p}^2_{il}(T), s^{l,2}_{p_i}(T)\}$ to be $\{\mathcal{D}_p\}^2$. Consequently, we provide the general error upper bound of updating weights in any $m$th iteration by recursion in the main theorem. 

\begin{theorem} \label{th:final}
For arbitrary iteration $m$, we can find the corresponding coefficient sets $\{\mathcal{D}^m_r\}^5_{r=1}$, $\{\tilde{\mathcal{C}}^m_r\}^3_{r=1}$, $\{\mathcal{B}^m_r\}^3_{r=1}$, $\{\mathcal{F}^m_r\}^3_{r=1}$ such that the 
deviation of updated weights in BFCNN from that in FCNN caused by our realization method is covered by the error upper bound
\begin{gather}
    \begin{split}
 \begin{pmatrix}
 \tilde{\mathcal{E}\mathcal{R}}(T) \\
    \tilde{\mathcal{E}\mathcal{S}}(T)
\end{pmatrix}_m
& \leq
\sum^{m-2}_{j=0} \{ 
\begin{pmatrix}
   \tilde{\mathcal{R}} \\
   \tilde{\mathcal{S}}
\end{pmatrix}_{m-j} \cdot \prod^m_{i=m-j+1} (\left[BFD\right]_i + \tilde{\mathcal{P}})\} \\
& + \prod^m_{r=2} (\left[BFD\right]_r +\tilde{\mathcal{P}})\cdot 
\begin{pmatrix}
    \tilde{\mathcal{R}} \\
    \tilde{\mathcal{S}}
\end{pmatrix}_1
    \end{split}
    \label{eq:final_estimation}
\end{gather}
with $\tilde{\mathcal{E}\mathcal{R}}(T)_m =\Vert w^{+,m}_1(T)-w^{-,m}_1(T) - \mathcal{W}^m_1 \Vert_{m}$, $\tilde{\mathcal{E}\mathcal{S}}(T)_m =\Vert w^{+,m}_2(T)-w^{-,m}_2(T) - \mathcal{W}^m_2 \Vert_{m}$, and 
\begin{gather*}
\begin{pmatrix}
    \tilde{\mathcal{R}} \\
    \tilde{\mathcal{S}}
\end{pmatrix}_m
=
\begin{pmatrix}
    \{\tilde{\mathcal{R}}^{qe}\}^m \\
    \{\tilde{\mathcal{S}}^{qe}\}^m
\end{pmatrix},~~~
\tilde{P} =
\begin{pmatrix}
 1&2\tilde{p}\\
 0&1+2\tilde{p}
\end{pmatrix},~
\\
\left[BFD\right]_m
=
2\tilde{p}\begin{pmatrix}
     \sum^3_{r=1}(\mathcal{B})^m_rT^{r-1} & \sum^5_{r=4}5(\mathcal{D})^m_rT^{r-4} \\
    \sum^3_{r=1}(\mathcal{F})^m_rT^{r-1} &    \sum^5_{r=4}3(\mathcal{D})^m_rT^{r-4}
 \end{pmatrix}.
\end{gather*}
\end{theorem}
\begin{proof}
Firstly, for $m=2$, the current errors in $\mathcal{M}^{le}_1$ increase the exponential coefficients $\mathcal{L}^8_{\pm,\tilde{s}}, \mathcal{L}^7_{\pm,v}$ for $\tilde{s} \in I_s, v \in I_v$ in \eqref{eq:Prop51}, \eqref{eq:prop52} by the form $\epsilon(T^ae^{-bT})$. Since the deviation term is bounded, the current errors consistently maintain their order. Then, Lemma \ref{lem_multi} shows that tail items $\{\mathcal{D}_y\}^2, \{\mathcal{D}_p\}^2$ should be added  to form new accumulative errors. Hence, 
\begin{equation*}
    \begin{split}
\Vert par^{\pm,2}_1(T)-par^2_{\pm,1} \Vert_{m=2} & \leq \{\mathcal{L}^{qe,1}_{\pm}\}^2 + \tilde{p}\Vert w_2 \Vert^2 \\
& ~~~ + 5\tilde{p} \{\mathcal{D}_y\}^2 + 4\tilde{p} \{\mathcal{D}_p\}^2,\\
\Vert par^{\pm,2}_2(T)-par^2_{\pm,2} \Vert_{m=2} & \leq \{\mathcal{L}^{qe,2}_{\pm}\}^2+\tilde{p}\Vert w_2 \Vert^2 \\
& ~~~+ 3\tilde{p} \{\mathcal{D}_y\}^2 + \tilde{p} \{\mathcal{D}_p\}^2.
    \end{split}
\end{equation*}
Note that coefficients $\{\mathcal{L}^{3,1}, \mathcal{L}^{3,2}_{\pm}\}^2$ are perturbed upwards than that in $m=1$ by indirect lateral accumulation on pBCRN. By combining errors in $\mathcal{M}^{le}_2$, we obtain the final estimation
\begin{equation*}
    \begin{split}
\Vert w^{\pm,2}_{1}(T) - \mathcal{W}^2_{\pm,1} \Vert_{m=2} & < \Vert w^{\pm,2}_{1}(0) - \mathcal{W}^1_{\pm,1} \Vert_{m=2} + \{\tilde{\mathcal{R}}^{qe}_{\pm}\}^2 \\
& + \tilde{p}\Vert w_2 \Vert^2 + 5\tilde{p} \{\mathcal{D}_y\}^2 + 4\tilde{p} \{\mathcal{D}_p\}^2, \\
\Vert  w^{\pm,2}_{2}(T) - \mathcal{W}^2_{\pm,2} \Vert_{m=2} & < \Vert w^{\pm,2}_2(0) - \mathcal{W}^1_{\pm,2} \Vert _{m=2} +\{\tilde{\mathcal{S}}^{qe}_{\pm}\}^2 \\
& + \tilde{p}\Vert w_2 \Vert^2 + 3\tilde{p}\{\mathcal{D}_y\}^2 +  \tilde{p}\{\mathcal{D}_p\}^2.
    \end{split}
\end{equation*}
Next, we consider further iterations by establishing the recursion formula. Here, $\left[\{\tilde{\mathcal{R}}^{qe}\}^{1} + \Vert w_1 \Vert^1 \right]$, $\left[\{\mathcal{S}^{qe}\}^{1} + \Vert w_2 \Vert^1 \right]$, existing $\{\mathcal{D}_y\}^2$, $\{\mathcal{D}_p\}^2$, can represent the upper bounds of $\Vert w^{+}_{1}(T)- w^{-}_{1}(T) - \mathcal{W}^1_{1} \Vert_{m=1}$, $\Vert w^{+}_{2}(T)- w^{-}_{2}(T)-\mathcal{W}^1_{2} \Vert_{m=1}$, respectively. Indeed, the cumbersome tail items display how weight errors in the last iteration propagate in the current iteration. Denoting $\tilde{\mathcal{E}}^j(T)_m = \Vert w^{+,m}_{j}(T) - \mathcal{W}^m_{\pm,j} 
 \Vert + \Vert w^{-,m}_{j}(T) - \mathcal{W}^m_{\pm,j} \Vert$ with $j=\{1,2\}$, then $\{\mathcal{D}_y\}^m, \{\mathcal{D}_p\}^m$ can be expressed by substituting corresponding elements. There exist related coefficients to have the recursion formula for the $m$th tail item and $\tilde{\mathcal{E}}^j(T)_{m-1}$ by $\mathcal{A}^m_1 =  5 \{\mathcal{D}_y\}^m + 4 \{\mathcal{D}_p\}^m$, $\mathcal{A}^m_2 =  3 \{\mathcal{D}_y\}^m + \{\mathcal{D}_p\}^m$, i.e.,
\begin{equation*}
    \begin{split}
\mathcal{A}^m_1 & = \sum^3_{m=1} \mathcal{B}^m_r \tilde{\mathcal{E}}^1(T)_{m-1} T^{r-1} + \sum^5_{r=4} 5 \mathcal{D}^m_r  \tilde{\mathcal{E}}^2(T)_{m-1} T^{r-4}, \\
\mathcal{A}^m_2 & = \sum^3_{r=1} \mathcal{F}^m_r \tilde{\mathcal{E}}^1(T)_{m-1} T^{r-1} + \sum^5_{r=4} 3 \mathcal{D}^m_r  \tilde{\mathcal{E}}^2(T)_{m-1} T^{r-4}, 
    \end{split}
\end{equation*}
where $\{\mathcal{D}_y\}^1 = \{\mathcal{D}_p\}^1 =0$ and
\begin{equation*}
\mathcal{B}^m_r =\left\{
\begin{aligned}
5 \mathcal{D}^m_r + 4\mathcal{C}^m_r, &r=1,2 \\
5 \mathcal{D}^m_r~~~~~, &r=3~~
\end{aligned},
\right.
~
\mathcal{F}^m_r =\left\{
\begin{aligned}
3 \mathcal{D}^m_r + \mathcal{C}^m_r,& r=1,2 \\
3 \mathcal{D}^m_r ~~~~~,& r=3
\end{aligned}.
\right.
\end{equation*}
Then, denoting errors of real updated weights in the $m$th iteration as $\tilde{\mathcal{E}\mathcal{R}}(T)_m$, $\tilde{\mathcal{E}\mathcal{S}}(T)_m$, we have a recurrence formula of their upper bounds
\begin{gather*}
\begin{split}
& \begin{pmatrix}
   \tilde{\mathcal{E}}^1(T) \\
   \tilde{\mathcal{E}}^2(T)
\end{pmatrix}_m
=
\begin{pmatrix}
    \{\tilde{\mathcal{R}}^{qe}\}^m\\
    \{\tilde{\mathcal{S}}^{qe}\}^m
\end{pmatrix}
+
\begin{pmatrix}
    \Vert w_1 \Vert^m+2\tilde{p}\Vert w_2 \Vert^m \\
    (1+2\tilde{p})\Vert w_2 \Vert^m 
\end{pmatrix} \\
& +
2\tilde{p} 
\begin{pmatrix}
    \sum^3_{r=1}(\mathcal{B})^m_rT^{i-1} & \sum^5_{r=4}5(\mathcal{D})^m_rT^{r-4} \\
    \sum^3_{r=1}(\mathcal{F})^m_rT^{r-1} &
    \sum^5_{r=4}3(\mathcal{D})^m_rT^{r-4}
\end{pmatrix}
\begin{pmatrix}
   \tilde{\mathcal{E}}^1(T) \\
   \tilde{\mathcal{E}}^2(T)
\end{pmatrix}_{m-1}
\end{split}
\end{gather*}
since $\tilde{\mathcal{E}\mathcal{R}}(T)_m \leq \tilde{\mathcal{E}}^1(T)_m$, $\tilde{\mathcal{E}\mathcal{S}}(T)_m \leq \tilde{\mathcal{E}}^2(T)_m$.
Consequently, we can derive the general upper bound by utilizing 
$\Vert w_1 \Vert^{m+1} = \tilde{\mathcal{E}}^1(T)_m$, $\Vert w_2 \Vert^{m+1} = \tilde{\mathcal{E}}^2(T)_m$. 
\end{proof}

The upper bound of the real updated weights in \eqref{eq:final_estimation} increases over the iteration numbers due to both the accumulative components and the increasingly large coefficients. However, this can be mitigated by extending the phase length $T$. Note that growing the lengthening $T$ results in a longer runtime, presenting a trade-off between accuracy and efficiency when running BFCNN. In our design, the upper bounds for calculations comprise entries of the form $T^a e^{-bT}$ for $a>0, b>0$. This implies that BFCNN exhibits the exponential-order convergence of errors in the finite $m$th iteration, as indicated by lemma \ref{eq:lem_for_ expone}, thereby allowing for enhanced precision at a relatively low cost of efficiency. Furthermore, due to the concrete description of powers outlined in \eqref{eq:estiW_1} and \eqref{eq:estiW_2}, augmenting the magnitude of the parameters in the FCNN, such as the training sample data, the initial weight values, within a valid range serves to amplify $\mathcal{N}^m, \tilde{\mathcal{N}}^m$. It consequently increases $\alpha^m, \alpha^m_3, n^m, \tilde{n}^m$ offering an alternative approach to decrease errors. Finally, for a more precise realization, the reaction rates can be accelerated by employing increasing the rate constant $k$ or the amplitude of the oscillator, without causing scale separations between modules.

\section{Application to binary classification problems} 
\label{sec:numeri}
To visualize the overall trend of the general upper bound of weight update errors (referred to as ``realization errors'' henceforth) concerning phase lengths and iteration numbers, we fixed the amplitude of the oscillator as one and selected a suitable group of phase lengths by adjusting the rate constant of oscillatory reactions for executing BFCNN multiple times. Then, we recorded the upper bound of realization errors $c\Vert w^{+,m}(T)- \mathcal{W}^m_{+} \Vert + c\Vert w^{-,m}(T)-\mathcal{W}^m_{-} \Vert $ with $c>0$ on logical classification problems ``OR'' and ``XOR'' at each iteration, where $\mathcal{W}^m_{\pm}$ are obtained in the reference module, FCNN, with the same initial positive/negative weights, training data and dual-rail computations implemented using MATLAB. 



\begin{figure*}
\centering  
\includegraphics[width=\textwidth]{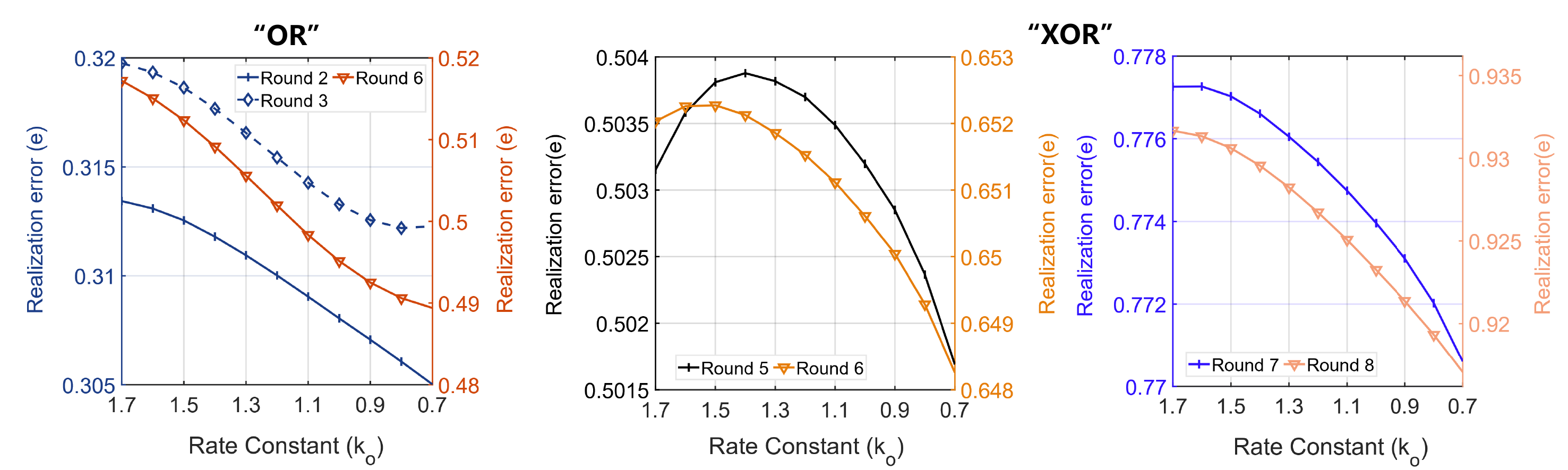}
\caption{Realization error change versus the finite phase length from $T= 21.481 (k_o = 1.7)$ units to $T= 49.289 (k_o=0.7)$ units of the chemical oscillator. The left picture is for the ``OR'' case and the right two are for the ``XOR'' case. }
     \label{fig.10}
\end{figure*}

Realization errors change versus phase length $T$ is divided into two cases, as shown in Fig.\ref{fig.10}. Since the training was finished after the fifth training round and no weight update occurred from the third to the fourth round when training ``OR'', we plotted realization error change curves in the second, third, and sixth rounds, respectively. This illustrates the decreasing nature of $\tilde{\mathcal{E}
\mathcal{R}}(T), \tilde{\mathcal{E}
\mathcal{S}}(T)$ with increasing $T$. For the ``XOR'' case, we selected four representative training round numbers to show the changing characteristics. Although the curves in the fifth and sixth rounds exhibit overshooting and not monotonically decreasing, the general tendency of realization error is to be reduced without contradicting the upper bound estimation for realization errors. Furthermore, the error cumulative phenomenon regarding iteration numbers can be observed in Fig.\ref{fig.10}, where the y-axis ranges of error curves with larger round numbers always lie above those of smaller round numbers. Finally, more concrete demonstrations are represented in Fig.\ref{fig.11} that are generated by the error data from the first round to the convergence round under certain finite phase lengths, confirming cumulative incremental trends.

\begin{figure*}
\centering  
\includegraphics[width=0.95\textwidth]{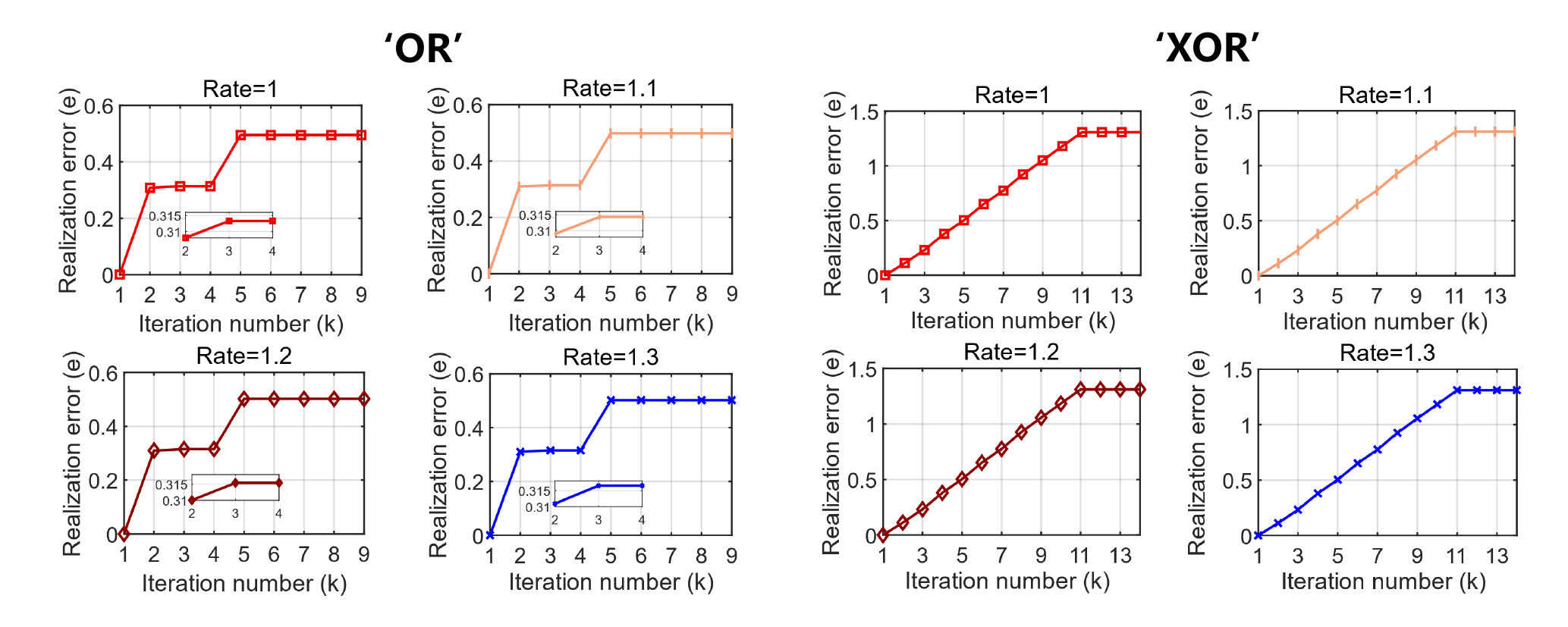}
\caption{Realization error change versus the iteration numbers from the first round to the convergence round. The left picture is for the ``OR'' case and the right one is for the ``XOR'' case. }
     \label{fig.11}
\end{figure*}

\section{Conclusion}
\label{sec:conclusion}
Part \Rmnum{2} has constructed a comprehensive framework for analyzing errors of the biochemical neural network, BFCNN, introduced in Part \Rmnum{1}, arising from representing computing results utilizing the finite approximation of LSSs, along with their vertical propagation and horizontal accumulation. In particular, this approach derives the general error upper bound of training weights in each iteration by examining errors in the first iteration to clarify the vertical propagation pathway and combining the influence from previous iterations to identify the horizontal propagation pathway. Our framework addresses a crucial gap in considering the impact of the occurrence and propagation of LSSs approximation errors on the performance of biological circuits designed by BCRN realization, which has the potential to offer effective ways to reduce errors.

A promising avenue for further research may concern enhancing the robustness of computational modules in such biological algorithm systems by designing controllers to conquer disturbances from accumulative errors, ensuring that equilibrium is maintained at a constant level. Additionally, inaccuracies caused by non-zero concentrations in the ``turn off'' phase of oscillators should be appended to establish the new error format in the future.

\section*{References}

\bibliographystyle{ieeetr}
\bibliography{IEEEREF}  

\begin{thebibliography}{10}

\bibitem{qian2018programming}
Y.~Qian, C.~McBride, and D.~Del~Vecchio, ``Programming cells to work for us,'' {\em Annual Review of Control, Robotics, and Autonomous Systems}, vol.~1, pp.~411--440, 2018.

\bibitem{qian2021robustness}
Y.~Qian and D.~Del~Vecchio, ``Robustness of networked systems to unintended interactions with application to engineered genetic circuits,'' {\em IEEE Transactions on Control of Network Systems}, vol.~8, no.~4, pp.~1705--1716, 2021.

\bibitem{xiao2018robust}
F.~Xiao and J.~C. Doyle, ``Robust perfect adaptation in biomolecular reaction networks,'' in {\em 2018 IEEE conference on decision and control (CDC)}, pp.~4345--4352, IEEE, 2018.

\bibitem{aoki2019universal}
S.~K. Aoki, G.~Lillacci, A.~Gupta, A.~Baumschlager, D.~Schweingruber, and M.~Khammash, ``A universal biomolecular integral feedback controller for robust perfect adaptation,'' {\em Nature}, vol.~570, no.~7762, pp.~533--537, 2019.

\bibitem{frei2022genetic}
T.~Frei, C.-H. Chang, M.~Filo, A.~Arampatzis, and M.~Khammash, ``A genetic mammalian proportional--integral feedback control circuit for robust and precise gene regulation,'' {\em Proceedings of the National Academy of Sciences}, vol.~119, no.~00, p.~e2122132119, 2022.

\bibitem{filo2022hierarchy}
M.~Filo, S.~Kumar, and M.~Khammash, ``A hierarchy of biomolecular proportional-integral-derivative feedback controllers for robust perfect adaptation and dynamic performance,'' {\em Nature communications}, vol.~13, no.~1, p.~2119, 2022.

\bibitem{qian2011neural}
L.~Qian, E.~Winfree, and J.~Bruck, ``Neural network computation with dna strand displacement cascades,'' {\em nature}, vol.~475, no.~7356, pp.~368--372, 2011.

\bibitem{cherry2018scaling}
K.~M. Cherry and L.~Qian, ``Scaling up molecular pattern recognition with dna-based winner-take-all neural networks,'' {\em Nature}, vol.~559, no.~7714, pp.~370--376, 2018.

\bibitem{zhang2020cancer}
C.~Zhang, Y.~Zhao, X.~Xu, R.~Xu, H.~Li, X.~Teng, Y.~Du, Y.~Miao, H.-c. Lin, and D.~Han, ``Cancer diagnosis with dna molecular computation,'' {\em Nature nanotechnology}, vol.~15, no.~8, pp.~709--715, 2020.

\bibitem{soloveichik2010dna}
D.~Soloveichik, G.~Seelig, and E.~Winfree, ``Dna as a universal substrate for chemical kinetics,'' {\em Proceedings of the National Academy of Sciences}, vol.~107, no.~12, pp.~5393--5398, 2010.

\bibitem{fages2017strong}
F.~Fages, G.~Le~Guludec, O.~Bournez, and A.~Pouly, ``Strong turing completeness of continuous chemical reaction networks and compilation of mixed analog-digital programs,'' in {\em Computational Methods in Systems Biology: 15th International Conference, CMSB 2017, Darmstadt, Germany, September 27--29, 2017, Proceedings 15}, pp.~108--127, Springer, 2017.

\bibitem{anderson2021reaction}
D.~F. Anderson, B.~Joshi, and A.~Deshpande, ``On reaction network implementations of neural networks,'' {\em Journal of the Royal Society Interface}, vol.~18, no.~177, p.~20210031, 2021.

\bibitem{arredondo2022supervised}
D.~Arredondo and M.~R. Lakin, ``Supervised learning in a multilayer, nonlinear chemical neural network,'' {\em IEEE Transactions on Neural Networks and Learning Systems}, 2022.

\bibitem{blount2017feedforward}
D.~Blount, P.~Banda, C.~Teuscher, and D.~Stefanovic, ``Feedforward chemical neural network: An in silico chemical system that learns xor,'' {\em Artificial life}, vol.~23, no.~3, pp.~295--317, 2017.

\bibitem{lakin2023design}
M.~R. Lakin, ``Design and simulation of a multilayer chemical neural network that learns via backpropagation,'' {\em Artificial Life}, pp.~1--28, 2023.

\bibitem{samaniego2021signaling}
C.~C. Samaniego, A.~Moorman, G.~Giordano, and E.~Franco, ``Signaling-based neural networks for cellular computation,'' in {\em 2021 American Control Conference (ACC)}, pp.~1883--1890, IEEE, 2021.

\bibitem{vasic2020crn++}
M.~Vasi{\'c}, D.~Soloveichik, and S.~Khurshid, ``Crn++: Molecular programming language,'' {\em Natural Computing}, vol.~19, no.~2, pp.~391--407, 2020.

\bibitem{vasic2022programming}
M.~Vasi{\'c}, C.~Chalk, A.~Luchsinger, S.~Khurshid, and D.~Soloveichik, ``Programming and training rate-independent chemical reaction networks,'' {\em Proceedings of the National Academy of Sciences}, vol.~119, no.~24, p.~e2111552119, 2022.

\bibitem{vasic2020deep}
M.~Vasic, C.~Chalk, S.~Khurshid, and D.~Soloveichik, ``Deep molecular programming: a natural implementation of binary-weight relu neural networks,'' in {\em International Conference on Machine Learning}, pp.~9701--9711, PMLR, 2020.

\bibitem{FAN20227}
Y.~Fan, X.~Zhang, and C.~Gao, ``Towards programming adaptive linear neural networks through chemical reaction networks,'' {\em IFAC-PapersOnLine}, vol.~55, no.~18, pp.~7--13, 2022.
\newblock 4th IFAC Workshop on Thermodynamics Foundations of Mathematical Systems Theory TFMST 2022.

\bibitem{fan2023automatic}
Y.~Fan, X.~Zhang, C.~Gao, and D.~Dochain, ``Automatic implementation of neural networks through reaction networks--part i: Circuit design and convergence analysis,'' {\em arXiv preprint arXiv:2311.18313}, 2023.

\bibitem{feinberg1972complex}
M.~Feinberg, ``Complex balancing in general kinetic systems,'' {\em Archive for rational mechanics and analysis}, vol.~49, no.~3, pp.~187--194, 1972.

\bibitem{feinberg1979lectures}
M.~Feinberg, ``Lectures on chemical reaction networks,'' {\em Notes of lectures given at the Mathematics Research Center, University of Wisconsin}, vol.~49, 1979.

\bibitem{horn1972general}
F.~Horn and R.~Jackson, ``General mass action kinetics,'' {\em Archive for rational mechanics and analysis}, vol.~47, no.~2, pp.~81--116, 1972.

\end{thebibliography}

\end{document}